\documentclass[11pt,a4paper,english]{article}

\newcommand{\beq}{\begin{equation}}
\newcommand{\eeq}{\end{equation}}

\usepackage[latin1]{inputenc}
\usepackage{babel}
\usepackage[centertags]{amsmath}
\usepackage{amsfonts}
\usepackage{amssymb}
\usepackage{color}
\usepackage{amsthm}
\usepackage{newlfont}
\usepackage{graphicx}
\usepackage{dsfont}
\usepackage{longtable}

\newtheorem{theorem}{Theorem}[section]
\newtheorem{lemma}[theorem]{Lemma}
\newtheorem{corollary}[theorem]{Corollary}
\newtheorem{proposition}[theorem]{Proposition}

\newtheorem{remark}[theorem]{Remark}
\newtheorem{Assumptions}[theorem]{Assumption}

\newcommand{\RR}{ \mathbb{R}}
\newcommand{\PP}{ \mathbb{P}}

\makeatletter  
\@addtoreset{equation}{section}
\def\theequation{\arabic{section}.\arabic{equation}}
\def \ep{\hbox{ }\hfill$\Box$}
\makeatother  
\begin{document} 
\title{\textbf{On an Optimal Extraction Problem with Regime Switching}\footnote{Financial support by the German Research Foundation (DFG) through the Collaborative Research Centre 1283 ``Taming uncertainty and profiting from randomness and low regularity in analysis, stochastics and their applications'' is gratefully acknowledged by the first author. The secod named author thanks the financial support by the China Scholarship Council (CSC).}}
\author{Giorgio Ferrari\thanks{Corresponding author. Center for Mathematical Economics, Bielefeld University, Germany; \texttt{giorgio.ferrari@uni-bielefeld.de}}
\and Shuzhen Yang \thanks{Institution of Financial Studies, Shandong University, P.R.C.; \texttt{yangsz@sdu.edu.cn}}}
\date{\today}
\maketitle

\vspace{0.5cm}

{\textbf{Abstract.}} This paper studies a finite-fuel two-dimensional degenerate singular stochastic control problem under regime switching that is motivated by the optimal irreversible extraction problem of an exhaustible commodity. A company extracts a natural resource from a reserve with finite capacity, and sells it in the market at a spot price that evolves according to a Brownian motion with volatility modulated by a two-state Markov chain. In this setting, the company aims at finding the extraction rule that maximizes its expected discounted cash flow, net of the costs of extraction and maintenance of the reserve. We provide expressions both for the value function and for the optimal control. On the one hand, if the running cost for the maintenance of the reserve is a convex function of the reserve level, the optimal extraction rule prescribes a Skorokhod reflection of the (optimally) controlled state process at a certain state and price dependent threshold. On the other hand, in presence of a concave running cost function it is optimal to instantaneously deplete the reserve at the time at which the commodity's price exceeds an endogenously determined critical level. In both cases, the threshold triggering the optimal control is given in terms of the optimal stopping boundary of an auxiliary family of perpetual optimal selling problems with regime switching. 
\smallskip


\smallskip

{\textbf{MSC2010 subsject classification}}: 93E20, 60G40, 49L20, 60J27, 91G80, 91B76

\smallskip


\section{Introduction}
\label{introduction}

Since the seminal work \cite{BrennanSchwartz}, both the literature in Applied Mathematics and that in Economics have seen numerous papers on optimal extraction problems of non-renewable resources under uncertainty. Some of these works formulate the extraction problem as an optimal timing problem (see, e.g., \cite{DixitPindyck}, \cite{Trigeorgis} and references therein); some as a combined absolutely continuous/impulse stochastic control problem (e.g., \cite{BrekkeOksendal} and \cite{LumleyZervos}); and some others as a stochastic optimal control problem only with classical absolutely continuous controls (cf.\ \cite{Insley1} and \cite{Feliz}, among many others), but with commodity price dynamics possibly described by a Markov regime switching model (cf., e.g., \cite{Insley2}). The latter kind of dynamics, firstly introduced in \cite{Hamilton}, may indeed help to explain boom and bust periods of commodity prices in terms of different regimes in a unique stochastic process.

In this paper we provide the solution to a stochastic irreversible extraction problem in presence of regime shifts in the underlying commodity spot price process. The problem we have in mind is that of a company extracting continuously in time a commodity from a reserve with finite capacity, and selling the natural resource in the spot market. The reserve level can be decreased at any time at a given proportional cost, following extraction policies which do not need to be rates. Moreover, the company faces a running cost (e.g.\ a cost for the reserve's maintenance) that is dependent on the reserve level. The company aims at finding the extraction rule that maximizes the expected discounted net cash flow in presence of market uncertainty and macroeconomic cycles. The latter are described through regime shifts in the volatility of the commodity spot price dynamics.

We set up the optimal extraction problem as a finite-fuel two-dimensional degenerate singular stochastic control problem under Markov regime switching. It is two-dimensional because for any regime $i$ the state variable consists of the value of the spot price, $x$, and the level of the reserve, $y$. It is a problem of singular stochastic control with finite fuel since extraction does not need to be performed at rates, and the commodity reserve has a finite capacity. Finally, it is degenerate since the state variable describing the level of the reserve is purely controlled, and does not have any diffusive component. 

While the literature on optimal stopping problems under regime switching is relatively rich (see, e.g., \cite{Bensoussanetal}, \cite{Buffington}, \cite{Guo01}, \cite{GuoZhang}, \cite{ZhangZhu}, among others), that on singular stochastic control problems with regime switching is still limited. We refer, e.g., to \cite{Pistorius}, \cite{Jiang}, \cite{CadSot} and \cite{ZhuYang} where the optimal dividend problem of actuarial science is formulated as a one-dimensional problem under Markov regime switching. If we then further restrict our attention to singular stochastic control problems with a two-dimensional state space and regime shifts, to the best of our knowledge \cite{Guoetal} is the only other paper available in the literature. That work addresses an optimal irreversible investment problem in which the growth and the volatility of the decision variable jump between two states at independent exponentially distributed random times. However, although in \cite{Guoetal} the authors provide a detailed discussion on the structure of the candidate solution and on the economic implications of regime switching for capital accumulation and growth, they do not confirm their guess by a verification theorem.



In this paper, with the aim of a complete analytical study, we assume that the commodity spot price $X$ evolves according to a Bachelier model\footnote{The choice of an arithmetic dynamics might be justified also at the modeling stage. Indeed, it has been shown in \cite{Geman} that for certain commodities an arithmetic dynamics fits historical time series better than a mean-reverting one. Moreover, it has been recently observed that some commodities can be traded at negative prices (see \cite{Fenton}). This happened, e.g., to propane prices in Edmonton (Canada) in June 2015.} with regime switching between two states. We show that the optimal extraction rule is of threshold type, and we provide the expression of the value function. 

The Hamilton-Jacobi-Bellman (HJB) equation associated to the optimal extraction problem takes the form of a system of two coupled variational inequalities with state dependent gradient constraints. The coupling is through the transition rates of the underlying continuous-time Markov chain $\varepsilon$, and it makes the problem of finding an explicit solution much harder than in the standard case without regime switching. We associate to the singular control problem a family of auxiliary optimal stopping problems for the Markov process $(X,\varepsilon)$. Such family is parametrized through the initial reserve level $y$. We solve the related free-boundary problem, and we characterize the geometry of stopping and continuation regions. As it is usual in optimal stopping theory, we show that the first time at which the underlying process leaves the continuation region is an optimal stopping rule. For any given and fixed $y$, such time takes the form of the first hitting time of $X$ to a regime dependent boundary $x^*_{i}(y)$, $i=1,2$. These boundaries are the unique solutions to a system of nonlinear algebraic equations derived by imposing the smooth-fit principle. 

Under the assumption that the running cost function is either strictly convex or concave in the reserve level, we show that the value function of the optimal extraction problem is given in terms of the value function of the auxiliary (family of) optimal stopping problems. Moreover, we prove that the optimal extraction policy is triggered by the optimal stopping boundaries $x^*_{i}(y)$, $i=1,2$. However, the behavior of the optimal control, and the regularity of the value function, significantly change when passing from a strictly convex running cost to a concave one.

On the one hand, if the running cost is a strictly convex function of the reserve level, we show that the optimal extraction policy keeps at any time the optimally controlled reserve level below a certain critical value $b^*$ with minimal effort, i.e.\ according to a \emph{Skorokhod reflection}. Such threshold depends on the spot price and on the market regime, and it is the inverse of the optimal stopping boundary $x^*_{i}(\,\cdot\,)$ previously determined. Also, we prove that, for any regime $i=1,2$, the value function of the optimal extraction problem is a $C^{2,1}$-solution to the associated HJB equation, and it is given as the integral, with respect to the controlled state variable, of the value function of the auxiliary optimal stopping problem. 

On the other hand, if the running cost is a concave function of the reserve level, the optimal extraction rule prescribes the instantaneous depletion of the reserve at the time at which the commodity's price in regime $i=1,2$ exceeds the critical level $x^*_{i}(y)$. As a consequence of such \emph{bang-bang} nature of the optimal policy - not extract or extract all - for any regime $i=1,2$ the value function only belongs to the class $C^0(\mathbb{R} \times [0,1]) \cap C^{1,1}(\mathbb{R} \times (0,1])$, with second order derivative with respect to $x$ that is bounded on any compact subset of $\mathbb{R} \times (0,1]$.

Although optimal controls of reflecting and bang-bang type already appeared in the literature on two-dimensional degenerate singular stochastic control problems (see, e.g., the recent \cite{DeAFeMo15}, \cite{DeAFeMo15bis} and references therein), to the best of our knowledge this is the first paper in which these two different behaviors of the optimal control arise in a model with Markov regime switching.

The study of the auxiliary family of optimal stopping problems performed in this paper is of interest on its own as well. Each stopping problem takes indeed the form of a perpetual optimal selling problem under regime switching that we completely solve. It is worth noticing that most of the papers dealing with optimal stopping problems with regime switching, and following a guess and verify approach, assume existence of a solution to the smooth-fit equations and additional properties of the candidate value function in order to perform a verification theorem (see, e.g., Theorem 3.1 in \cite{GuoZhang}, and Theorems 3 and 5 in \cite{ZhangZhu}). An abstract and nonconstructive approach, based on a thorough analysis of the related variational inequality, is adopted in \cite{Bensoussanetal}. Here, instead, we construct a solution to the free-boundary problem, and we then prove all the properties needed to verify that such solution is actually the value function of our optimal stopping problem with regime switching (see our Theorems \ref{candidate-w} and \ref{thm:verifying} below). We believe that also such a result represents an interesting contribution to the literature.

Although not solvable in closed form, the system of nonlinear algebraic equations characterizing the optimal stopping boundaries - hence the optimal extraction policy - can be easily solved numerically. This fact allows us to compare the optimal extraction boundaries in the case with and without regime switching, and thus to draw interesting economic conclusions (see Section \ref{comparison}). 
In particular, we show that in presence of macroeconomic cycles, the company is more reluctant (resp.\ favourable) to extract and then sell the commodity, relative to the case in which the market were always in the good (resp.\ bad) regime with the lowest (resp.\ highest) volatility. 


The rest of the paper is organized as follows. In Section \ref{sec:problem} we formulate the optimal extraction problem, we introduce the associated HJB equation, and we discuss the solution approach. The family of optimal stopping problems is then solved in Section \ref{sec:OS}, whereas the optimal control is provided in Section \ref{sec:solSSC}. A comparison with the optimal extraction rule that one would find in the no-regime-switching case, as well as some economic conclusions, are contained in Section \ref{comparison}. Appendix \ref{someproofs} collects the proofs of some results of Section \ref{sec:OS}, whereas in Appendix \ref{app} one can find auxiliary results needed in the paper.


\section{Problem Formulation and Solution Approach}
\label{sec:problem}

\subsection{The Optimal Extraction Problem}

Let $(\Omega,\mathcal{F}, \mathbb{P})$ be a complete probability space, rich enough to accommodate a one-dimensional Brownian motion $\{W_t, t\geq 0\}$ and a continuous-time Markov chain $\{\varepsilon_t, t\geq 0\}$ with state space $E:=\{1,2\}$, and with irreducible generator matrix 
\begin{equation}
\label{Lambdamatrix} 
Q:=
\begin{pmatrix}
-\lambda_1 & \lambda_1 \\
\lambda_2 & -\lambda_2 
\end{pmatrix},
\end{equation}
for some $\lambda_1, \lambda_2 > 0$. The Markov chain $\varepsilon$ jumps between the two states at exponentially distributed random times, and the constant $\lambda_i$ gives the rate of leaving state $i=1,2$. We take $\varepsilon$ independent of $W$ and denote by $\mathbb{F}:=\{\mathcal{F}_t, t\geq 0\}$ the filtration jointly generated by $W$ and $\varepsilon$, as usual augmented by $\mathbb{P}$-null sets.

We assume that the spot price of the commodity evolves according to a Bachelier model \cite{Bachelier} with regime switching; i.e.\
\beq
\label{dyn:X}
dX_t = \sigma_{\varepsilon_t} dW_t, \quad t>0, \qquad X_0=x \in \mathbb{R},
\eeq
where for every state $i=1,2$ $\sigma_i > 0$ is a known finite constant. From the modeling point of view, the choice of an arithmetic dynamics might be justified by noticing that certain commodities can be traded at negative spot prices (see, e.g., \cite{Fenton}), and do not show a mean-reverting behavior (cf.\ \cite{Geman}, among others).

$(X,\varepsilon)$ is a strong Markov process (see \cite{ZhuYin}, Remark 3.11) and we set $\mathbb{P}_{(x,i)}(\,\cdot\,):=\mathbb{P}(\,\cdot\,| X_0=x, \varepsilon_0=i)$, and we denote by $\mathbb{E}_{(x,i)}$ the corresponding expectation operator. From Section $3.1$ in \cite{ZhuYin} we also know that $(X,\varepsilon)$ is regular, in the sense that the sequence of stopping times $\{\beta_n, n\in \mathbb{N}\}$, with $\beta_n:=\inf\{t\geq 0: |X_t| =n\}$, is such that $\lim_{n\uparrow \infty} \beta_n = +\infty$, $\mathbb{P}_{(x,i)}$-a.s.

The level of the commodity reserve satisfies
\beq
\label{dyn:Y}
dY^{\nu}_t = -d\nu_t, \quad t>0, \qquad Y^{\nu}_0 = y \in [0,1].
\eeq
Taking $y \leq 1$ we model the fact that the reserve has a finite capacity, normalized to $1$ without loss of generality. Here $\nu_t$ represents the cumulative amount of commodity extracted up to time $t\geq 0$. We say that an extraction policy is admissible if, given $y \in [0,1]$, it belongs to the nonempty convex set
\begin{eqnarray}
\label{admissiblecontrols}
\mathcal{A}_y \hspace{-0.2cm}&:=&\hspace{-0.2cm} \{\nu:\Omega \times \mathbb{R}_{+} \mapsto  \mathbb{R}_{+}, ({\nu_{t}(\omega) := \nu(\omega,t)})_{t \geq 0}\mbox{ is nondecreasing,\,\,left-continuous,} \nonumber \\
&& \hspace{3cm} \mathbb{F}-\mbox{adapted with} \,\,y- \nu_t\geq 0\,\,\,\forall \; t \geq 0,\,\,\nu_0=0\,\,\,\,\PP-\mbox{a.s.}\}. 
\end{eqnarray}
Moreover, we let $\mathbb{P}_{(x,y,i)}(\,\cdot\,):=\mathbb{P}(\,\cdot\,| X_0=x, Y_0=y, \varepsilon_0=i)$ and $\mathbb{E}_{(x,y,i)}$ the corresponding expectation operator.

While extracting, the company faces two types of costs: the first one is an extraction cost that we take proportional through a constant $c>0$ to the amount of commodity extracted; the second one is a running cost, e.g.\ an holding cost for the maintenance of the reserve. The latter is measured by a function $f$ of the reserve level satisfying the following assumption.
\begin{Assumptions}
\label{Asscost}
$f:\mathbb{R} \to \mathbb{R}_+$ is increasing, continuous on $[0,1]$ and such that $f(0)=0$. Moreover, one of the following two conditions is satisfied:
\begin{itemize}
\item[(I)] $y \mapsto f(y)$ is strictly convex and continuously differentiable on $[0,1]$; 
\item[(II)] $y \mapsto f(y)$ is concave on $[0,1]$ and continuously differentiable on $(0,1]$.
\end{itemize}
\end{Assumptions}
\noindent Assumption \ref{Asscost} will be standing throughout this paper.

\begin{remark}
\begin{enumerate}\hspace{10cm}
\item From an economic point of view, a running cost function that is concave on $[0,1]$ reflects \rm{economies of scale} in the size of the operation. On the other hand, a running cost function convex on $[0,1]$ seems to be more appropriate for a company facing \rm{diseconomies of scale}. 
\item The requirement $f(0)=0$ is without loss of generality, since if $f(0)=f_o>0$ then one can always set $\hat{f}(y):=f(y)-f_o$ and write $f(y)=\hat{f}(y)+f_o$, so that the firms's optimization problem (cf.\ \eqref{value} below) remains unchanged up to an additive constant.
\item Cost functions of the form $f(y)=\alpha_o y^2 + \beta_o y$ for some $\alpha_o,\beta_o>0$, $f(y)=y^{\gamma_o}$, for some $\gamma_o \in (0,1)$, or $f(y)=\alpha y$ for $\alpha > 0$, clearly meet Assumption \ref{Asscost}.
\end{enumerate}
\end{remark}

Following an extraction policy $\nu \in \mathcal{A}_y$ and selling the extracted amount in the spot market at price $X$, the expected discounted cash flow of the company, net of extraction and maintenance costs, is
\beq
\label{functional}
\mathcal{J}_{x,y,i}(\nu):=\mathbb{E}_{(x,y,i)}\bigg[\int_0^{\infty} e^{-\rho t} \big(X_t - c\big)d\nu_t - \int_0^{\infty} e^{-\rho t}f(Y^{\nu}_t)dt\bigg],\qquad (x,y,i) \in \mathcal{O},
\eeq
where $\rho > 0$ is a given discount factor and $\mathcal{O}:=\mathbb{R}\times [0,1] \times \{1,2\}$. Throughout this paper, for $t > 0$ and $\nu \in \mathcal{A}_y$ we will make use of the notation $\int_0^{t} e^{-\rho s}(X_s-c) d\nu_s$ to indicate the Stieltjes integral $\int_{[0,t)} e^{-\rho s} (X_s-c) d\nu_s$ with respect to $\nu$. As a byproduct of Lemma \ref{lemma:UI} in Appendix \ref{app}, the functional \eqref{functional} is well-defined and finite for any $\nu \in \mathcal{A}_y$.

The company aims at choosing an admissible extraction rule that maximizes \eqref{functional}; that is, it faces the optimization problem
\beq
\label{value}
V(x,y,i) := \sup_{\nu \in \mathcal{A}_y}\mathcal{J}_{x,y,i}(\nu), \qquad (x,y,i) \in \mathcal{O}.
\eeq

\noindent Notice that if $y=0$ then no control can be exerted, i.e.\ $\mathcal{A}_0 = \{\nu\equiv 0\}$, and therefore $V(x,0,i)= \mathcal{J}_{x,0,i}(0) =0$, for any $(x,i) \in \mathbb{R} \times \{1,2\}$.

Problem \eqref{value} falls into the class of singular stochastic control problems, i.e.\ problems in which admissible controls do not need to be absolutely continuous with respect to the Lebesgue measure, as functions of time (see \cite{Shreve88} and Chapter VIII in \cite{FlemingSoner} for an introduction). In particular, it is a finite-fuel two-dimensional degenerate singular stochastic control problem under Markov regime switching. It is degenerate because the state process $Y$ is purely controlled, and does not have a diffusive component. Moreover, it is of finite-fuel type since the controls stay bounded.

\begin{remark}
\label{rem:concaveconvex}\hspace{10cm}
\begin{enumerate}
\item In the literature on optimal extraction it is common to consider the problem of a company maximizing the total expected profits, net of the total expected costs of extraction (see \cite{Insley2} and \cite{Pindyck81}, among others); that is, (in our formulation) maximizing $\mathbb{E}[\int_0^{\infty} e^{-\rho t} (X_t - c)d\nu_t]$. In \eqref{functional} we have also the term $\mathbb{E}[\int_0^{\infty} e^{-\rho t}f(Y^{\nu}_t)dt]$ in order to account for the possible running costs incurred by the company, e.g., for the maintenance of the reserve. However, as it is discussed in Remark \ref{linearcase}, our results carry over to the case $f\equiv 0$ as well.
\item Due to the convexity of $\mathcal{A}_y$, and the linearity of $\nu \mapsto Y^{\nu}$, if $y \mapsto f(y)$ is strictly convex on $[0,1]$, then the functional $\mathcal{J}_{x,y,i}(\,\cdot\,)$ is strictly concave on $\mathcal{A}_y$, and \eqref{value} is a well-posed maximization problem of a concave functional. On the other hand, if $y \mapsto f(y)$ is concave on $[0,1]$, then $\mathcal{J}_{x,y,i}(\,\cdot\,)$ is convex on $\mathcal{A}_y$. We will see in Section \ref{sec:solSSC} how the convexity/concavity of $f$ will impact on the behavior of the optimal control, and on the regularity of the value function.
\end{enumerate}
\end{remark}

\begin{remark}
\label{rem:priceimpact}
Since the extraction rule adopted by the company does not affect the price of the commodity, our model takes into consideration a price-taker company. Allowing for a direct instantaneous effect of the extraction policy on the price dynamics, our problem would share a similar mathematical structure with the problem of optimal execution in algorithm trading, where an investor sells a large number of stock shares over a given time horizon and her actions have impact on the stock price (see, e.g., \cite{GuoZervos} for a recent formulation of the optimal execution problem involving singular controls).
We leave the analysis of the optimal extraction problem with price impact as an interesting future research topic.
\end{remark}


\subsection{The Hamilton-Jacobi-Bellman Equation and a First Verification Theorem}
\label{sec:HJB}

In light of classical results in stochastic control (see, e.g., Chapter VIII in \cite{FlemingSoner}), we expect that for any $i=1,2$ the value function $V(\cdot,\cdot,i)$ suitably satisfies the Hamilton-Jacobi-Bellman (HJB) equation
\beq
\label{HJB}
\max\Big\{\big(\mathcal{G} - \rho\big)U(x,y,i) - f(y), (x-c) - U_y(x,y,i)\Big\} =0,
\eeq
for $(x,y) \in \mathbb{R} \times (0,1]$ and with boundary condition $U(x,0,i)=0$. 
Here $\mathcal{G}$ is the infinitesimal generator of $(X,\varepsilon)$. It acts on functions $h: \mathbb{R} \times \{1,2\} \to \mathbb{R}$ with $h(\cdot, i) \in C^2(\mathbb{R})$ for any given and fixed $i=1,2$ as
\beq
\label{generator}
\mathcal{G} h(x,i):= \frac{1}{2}\sigma^2_i h_{xx}(x,i) + \lambda_i \big(h(x,3-i) - h(x,i)\big).
\eeq
It is worth noting that, due to \eqref{generator}, equation \eqref{HJB} is actually a system of two variational inequalities with state-dependent gradient constraints, coupled through the transition rates $\lambda_1,\lambda_2$. The next preliminary verification result shows that any suitable solution to \eqref{HJB} provides an upper bound for the value function $V$.

\begin{theorem}
\label{1stverification}
For $i=1,2$, let $U(\cdot, \cdot, i) \in C^{1,1}(\mathbb{R} \times (0,1))$ be such that $U_{xx}(\cdot,\cdot,i) \in L^{\infty}_{loc}(\mathbb{R} \times (0,1))$, $U(x,0,i)=0$, $x \in \mathbb{R}$, and $|U(x,y,i)| \leq K(1+|x|)$, for any $(x,y) \in \mathbb{R} \times [0,1]$ and for some $K>0$. Then if $U$ solves \eqref{HJB} in the a.e.\ sense, one has $U \geq V$ on $\mathcal{O}$. 
\end{theorem}
\begin{proof}
Fix $(x,y,i)\in \mathcal{O}$, and take arbitrary $R>0$ and $T>0$. Set $\tau_{R}:=\inf\big\{t\ge 0\,:\,X_t\notin(-R,R)\big\}$, and let $0 \leq \eta_1 < \eta_2 < ... < \eta_{N} \leq \tau_R \wedge T$ be the random times of jumps of $\varepsilon$ in the interval $[0,\tau_R \wedge T)$ (clearly, the number $N$ of those jumps is random as well). Notice that by the regularity of $U$ we can approximate $U$ (uniformly on compact subsets of $\mathbb{R}\times (0,1)$) by a sequence of functions $\{U^{(m)}\}_{m\geq 1}$ such that $U^{(m)}(\cdot,\cdot,i)\in C^{\infty,1}(\mathbb{R} \times (0,1))$ for any $i=1,2$ (see, e.g., part (a) of the proof of Theorem 4.1 in Ch.\ VIII of \cite{FlemingSoner}, or the proof of Theorem 2.7.9 in \cite{KS-MF} for this kind of procedure). Then pick an admissible control $\nu$ and apply It\^o-Meyer's formula for semimartingales (\cite{Meyer}, pp.\ 278--301) to the process $(e^{-\rho t} U^{(m)}(X_{t},{Y}_{t}^{\nu}, \varepsilon_{t}))_{t\geq0}$ on each of the intervals $[0, \eta_1)$, $(\eta_1,\eta_2)$,...,$(\eta_N,\tau_R \wedge T)$. Piecing together all the terms as in the proof of Lemma 3 at p.\ 104 of \cite{Sk} (see also Lemma 2.4 in \cite{YinXi} for a similar idea of proof), and finally taking limits as $m\uparrow \infty$ one finds

\begin{align*}
U(x,y,i) =& \mathbb{E}_{(x,y,i)}\bigg[e^{-\rho(\tau_{R}\wedge T)}U(X_{\tau_{R}\wedge T}, {Y}_{\tau_{R}\wedge T}^{\nu}, \varepsilon_{\tau_{R}\wedge T})-\int_0^{\tau_{R}\wedge T}e^{-\rho s}(\mathcal{G}-\rho)
U(X,{Y}^{\nu}_s, \varepsilon_s)ds\bigg] \\
&  + \mathbb{E}_{(x,y,i)}\bigg[\int_0^{\tau_{R}\wedge T}e^{-\rho s}U_y(X_s,{Y}^{\nu}_s, \varepsilon_s)d\nu_s\bigg]
\\ &  - \mathbb{E}_{(x,y,i)}\Big[\sum_{0\leq s < \tau_{R}\wedge T}e^{-\rho s}
\left(U(X_s,{Y}^{\nu}_{s+}, \varepsilon_{s})-U(X_s,{Y}^{\nu}_s,\varepsilon_{s})-U_y(X_s,{Y}^{\nu}_s, \varepsilon_s)\Delta Y_s\right)\Big],
\end{align*}
where $\Delta Y_s := Y_{s+} - Y_s = -\Delta \nu_s := -(\nu_{s+} - \nu_s)$, and the expectation of the stochastic integral vanishes since $U_x$ is bounded on $(x,y,i)\in[-R,R]\times[0,1] \times \{1,2\}$.

Now, noticing that any admissible control $\nu$ can be written as the sum of its continuous part and of its pure jump part, i.e.\ $d\nu=d\nu^{cont}+ \Delta \nu$, one has 
\begin{align*}
U(x,y,i) = & \mathbb{E}_{(x,y,i)}\bigg[e^{-\rho(\tau_{R}\wedge T)}U(X_{\tau_{R}\wedge T}, Y_{\tau_{R}\wedge T}^{\nu}, \varepsilon_{\tau_{R}\wedge T}) - \int_0^{\tau_{R}\wedge T}e^{-\rho s}(\mathcal{G}-\rho)U(X_s,Y^{\nu}_s, \varepsilon_s)ds\bigg] \\
&  + \mathbb{E}_{(x,y,i)}\bigg[\int_0^{\tau_{R}\wedge T}e^{-\rho s}U_y(X_s,Y^{\nu}_s, \varepsilon_s)d\nu_s^{cont}\bigg] \\
& - \mathbb{E}_{(x,y,i)}\Big[\sum_{0\le s < \tau_{R}\wedge T}e^{-\rho s}
\left(U(X_s,Y^{\nu}_{s+}, \varepsilon_{s})-U(X_s,Y^{\nu}_{s}, \varepsilon_{s})\right)\Big].
\end{align*}
Because
\begin{equation}
\label{jump}
U(X_s,Y^{\nu}_{s+},\varepsilon_{s})-U(X_s, Y^{\nu}_{s},\varepsilon_{s})=
-\int_0^{\Delta \nu_s} U_y(X_s,Y^{\nu}_{s} - z , \varepsilon_s)dz,
\end{equation}
and since $U$ satisfies the HJB equation \eqref{HJB}, one obtains
\begin{align}
\label{verif04}
U(x,y,i) \geq &\mathbb{E}_{(x,y,i)}\left[e^{-\rho(\tau_{R}\wedge T)}U(X_{\tau_{R}\wedge T}, Y_{\tau_{R}\wedge T}^{\nu}, \varepsilon_{\tau_{R}\wedge T})\right]-\mathbb{E}_{(x,y,i)}\bigg[\int_0^{\tau_{R}\wedge T}e^{-\rho s} f(Y^{\nu}_s)ds\bigg]\nonumber \\
& + \mathbb{E}_{(x,y,i)}\bigg[\int_0^{\tau_{R}\wedge T}e^{-\rho s} (X_s-c) d\nu_s^{cont}\bigg]
 + \mathbb{E}_{(x,y,i)}\Big[\sum_{0\leq s < \tau_{R}\wedge T}e^{-\rho s}
(X_s-c) \Delta \nu_s \Big] \nonumber  \\
 = &\mathbb{E}_{(x,y,i)}\bigg[e^{-\rho(\tau_{R}\wedge T)}U(X_{\tau_{R}\wedge T}, Y_{\tau_{R}\wedge T}^{\nu}, \varepsilon_{\tau_{R}\wedge T}) + \int_0^{\tau_{R}\wedge T}e^{-\rho s}(X_s-c) d\nu_s\bigg] \\
 & - \mathbb{E}_{(x,y,i)}\bigg[\int_0^{\tau_{R}\wedge T}e^{-\rho s}f(Y^{\nu}_s)ds \bigg].\nonumber
\end{align}

By H\"older's inequality, \eqref{dyn:X}, and It\^o's isometry we have
\begin{eqnarray*}
\mathbb{E}_{(x,y,i)}\Big[e^{-\rho(\tau_{R}\wedge T)}|X_{\tau_{R}\wedge T}|\Big] & \hspace{-0.25cm} \leq \hspace{-0.25cm} & \mathbb{E}_{(x,y,i)}\Big[e^{-2\rho(\tau_{R}\wedge T)}\Big]^{\frac{1}{2}}\mathbb{E}_{(x,y,i)}\Big[|X_{\tau_{R}\wedge T}|^2\Big]^{\frac{1}{2}} \nonumber \\
& \hspace{-0.25cm} \leq \hspace{-0.25cm} & \sqrt{2}\mathbb{E}_{(x,y,i)}\Big[e^{-2\rho(\tau_{R}\wedge T)}\Big]^{\frac{1}{2}}\bigg(|x|^2 + \mathbb{E}_{(x,y,i)}\bigg[\Big|\int_0^{\tau_{R}\wedge T} \sigma_{\varepsilon_u} dW_u\Big|^2\bigg]\bigg)^{\frac{1}{2}} \nonumber \\
& \hspace{-0.25cm} \leq \hspace{-0.25cm} &  \sqrt{2}\mathbb{E}_{(x,y,i)}\Big[e^{-2\rho(\tau_{R}\wedge T)}\Big]^{\frac{1}{2}}\Big(|x|^2 + (\sigma_1^2 \vee \sigma_2^2)T\Big)^{\frac{1}{2}}.\nonumber
\end{eqnarray*}
The previous estimate, together with the linear growth property of $U$, then imply
\begin{eqnarray*}
&& \mathbb{E}_{(x,y,i)}\Big[e^{-\rho(\tau_{R}\wedge T)}U(X_{\tau_{R}\wedge T}, Y_{\tau_{R}\wedge T}^{\nu}, \varepsilon_{\tau_{R}\wedge T})\Big] \geq  - C \mathbb{E}_{(x,y,i)}\Big[e^{-\rho(\tau_{R}\wedge T)}\Big] \nonumber \\
&& - \sqrt{2} C \mathbb{E}_{(x,y,i)}\Big[e^{-2\rho(\tau_{R}\wedge T)}\Big]^{\frac{1}{2}}\Big(|x|^2 + (\sigma_1^2 \vee \sigma_2^2)T\Big)^{\frac{1}{2}},
\end{eqnarray*}
for some constant $C>0$. 
Hence
\begin{eqnarray}
\label{verif04-bis}
&& U(x,y,i) \geq - C \mathbb{E}_{(x,y,i)}\Big[e^{-\rho(\tau_{R}\wedge T)}\Big] - \sqrt{2} C \mathbb{E}_{(x,y,i)}\Big[e^{-2\rho(\tau_{R}\wedge T)}\Big]^{\frac{1}{2}}\Big(|x|^2 + (\sigma^2_1 \vee \sigma^2_2)T\Big)^{\frac{1}{2}} \nonumber \\
&& \hspace{+0.5cm} + \mathbb{E}_{(x,y,i)}\bigg[ \int_0^{\tau_{R}\wedge T}e^{-\rho s}(X_s-c) d\nu_s\bigg] -\mathbb{E}_{(x,y,i)}\bigg[\int_0^{\tau_{R}\wedge T}e^{-\rho s}f(Y^{\nu}_s)ds \bigg].
\end{eqnarray}
When taking limits as $R\to\infty$ we have $\tau_{R}\wedge T \rightarrow T$, $\mathbb{P}_{(x,y,i)}$-a.s.\ by regularity of $(X,\varepsilon)$. By Lemma \ref{lemma:UI} in Appendix \ref{app}, the integrals on the right-hand side of \eqref{verif04-bis} are uniformly integrable. We can thus invoke Vitali's convergence theorem to take limits as $R\uparrow\infty$ in \eqref{verif04-bis}, and then as $T\uparrow \infty$, and obtain
\begin{align}
\label{final-part1}
U(x,y,i) \geq \mathbb{E}_{(x,y,i)}\bigg[\int_0^{\infty}e^{-\rho s} (X_s-c) d\nu_s - \int_0^{\infty}e^{-\rho s}f(Y^{\nu}_s)ds \bigg].
\end{align}
Since \eqref{final-part1} holds for any $\nu \in \mathcal{A}_y$, we have $U(x,y,i) \geq V(x,y,i)$. Hence $U \geq V$ on $\mathcal{O}$ by arbitrariness of $(x,y,i) \in \mathcal{O}$.
\end{proof}


\subsection{The Solution Approach}

In this paper we solve problem \eqref{value} in the following two cases (cf.\ Assumption \ref{Asscost} and Remark \ref{rem:concaveconvex}): 
\begin{itemize}
\item[(I)] $y \mapsto f(y)$ is strictly convex on $[0,1]$ (cf.\ Section \ref{subsec:fconvex}); 
\item[(II)] $y \mapsto f(y)$ is concave on $[0,1]$ (cf.\ Section \ref{subsec:fconcave}). 
\end{itemize}
The case of a running cost that is neither convex nor concave on $[0,1]$ needs a separate analysis, and it is left as an interesting open problem (see the recent \cite{DeAFeMo15} and \cite{DeAFeMo15bis} for singular stochastic control problems in which the running cost is neither convex nor concave).

We will follow a \emph{guess-and-verify} approach, by finding in each of the two previous cases a suitable solution to \eqref{HJB}, and then verifying its optimality through a verification theorem. As a byproduct, we will also obtain the optimal control rule. We will see that in both cases (I) and (II) the solution to \eqref{value} is given in terms of the solution to the parameter-dependent (as $y\in (0,1]$ enters only as a parameter) optimal stopping problem with regime switching 
\beq
\label{def-u}
u(x,i;y):=\sup_{\tau \geq 0}\mathbb{E}_{(x,i)}\Big[e^{-\rho \tau}(X_{\tau}-\theta(y))\Big].
\eeq

In \eqref{def-u} the optimization is taken over all $\mathbb{P}_{(x,i)}$-a.s.\ finite $\mathbb{F}$-stopping times; moreover, $\theta(y)$ is a given suitable real number that depends on the initial level of the reserve, $y$, through the running cost function $f$. In particular,
\begin{align}
\label{thetadef}
\theta(y):=\left\{
\begin{array}{ll}
\displaystyle c-\frac{f'(y)}{\rho}\,\,\quad \mbox{if Case (I) holds}\\[+14pt]
\displaystyle c - \frac{1}{\rho}\frac{f(y)}{y}\quad \mbox{if Case (II) holds}.
\end{array}
\right.
\end{align}

To obtain an heuristic justification of the relation between problems \eqref{value} and \eqref{def-u} one can argue as follows.
On the one hand, formally differentiating \eqref{HJB} with respect to $y$ inside the region where $(\mathcal{G} - \rho)V(x,y,i) - f(y) =0$, one sees that for any $i=1,2$ $V_y$ should identify with an appropriate solution to the variational inequality
\begin{equation}
\label{HJB-OS}
\max\Big\{\big(\mathcal{G} - \rho\big) \zeta(x,i;y) - f'(y), x-c - \zeta(x,i;y)\Big\} =0,
\end{equation}
for $x \in \mathbb{R}$ and any given $y\in [0,1]$. 

As well as \eqref{HJB}, notice that also \eqref{HJB-OS} is actually a system of variational inequalities. In fact, it is the variational inequality associated to the family of optimal stopping problem with regime switching 
\begin{eqnarray}
\label{value-OS-heur}
& \displaystyle \sup_{\tau \geq 0}\mathbb{E}_{(x,i)}\Big[e^{-\rho \tau}\big(X_{\tau}-c\big) - \int_{0}^{\tau} e^{-\rho s} f'(y) ds\Big].  
\end{eqnarray}
By evaluating the time integral in \eqref{value-OS-heur}, we easily see that \eqref{value-OS-heur} rewrites as
$$ \sup_{\tau \geq 0}\mathbb{E}_{(x,i)}\Big[e^{-\rho \tau}\Big(X_{\tau}-c + \frac{f'(y)}{\rho}\Big)\Big] - \frac{f'(y)}{\rho},$$
which is clearly equivalent to \eqref{def-u} when $\theta(y)=c-\frac{f'(y)}{\rho}$.

A differential connection between the value functions of a singular control problem and of an optimal stopping problem is commonly observed in singular control problems in which the payoff functional to be maximized is concave with respect to the control variable (see, e.g., \cite{KaratzasBaldursson} and references therein). In light of Remark \ref{rem:concaveconvex} we then expect that $V_y=u$ in Case (I); i.e.\ when $f$ is (strictly) convex.

On the other hand, optimal stopping problem \eqref{def-u} can also arise if we restrict the optimization in \eqref{value} to all the controls of the following purely discontinuous bang-bang type: for some $\mathbb{F}$-stopping time $\tau$ and for any given $y \in [0,1]$, $\nu_t=0$ for any $t\leq \tau$, and $\nu_t=y$ for any $t>\tau$. Indeed, following such a policy, and optimizing with respect to the time of reserve's depletion $\tau$, one ends up with the optimal stopping problem
\begin{eqnarray*}
&\displaystyle \sup_{\tau \geq 0}\mathbb{E}_{(x,i)}\Big[e^{-\rho \tau}\big(X_{\tau}-c\big)y - \int_{0}^{\tau} e^{-\rho s} f(y) ds\Big],
\end{eqnarray*}
which easily rewrites as
$$y\,\sup_{\tau \geq 0}\mathbb{E}_{(x,i)}\Big[e^{-\rho \tau}\Big(X_{\tau}- c + \frac{1}{\rho}\frac{f(y)}{y}\Big)\Big] - \frac{f(y)}{\rho}.$$
The latter is clearly related to \eqref{def-u} when $\theta(y)=c - \frac{1}{\rho}\frac{f(y)}{y}$. 

We expect that a similar connection to problem \eqref{value} (and therefore the optimality of a policy prescribing the instantaneous depletion of the reserve at a suitable stopping time) holds in Case (II). Indeed, in such a case $f$ is concave, and therefore the marginal holding cost of the reserve decreases.

Supported by the previous heuristic discussion, in the next section we will solve problem \eqref{def-u} when $\theta(y)$ is a given constant. In particular, we will show that the solution to \eqref{def-u} is triggered by suitable regime-dependent stopping boundaries $x^*_i(y)$, $y\in (0,1]$, that we will characterize as the unique solutions to a system of nonlinear algebraic equations. These boundaries will then play a crucial role in the construction of the optimal control in both Case (I) and Case (II) (see Sections \ref{subsec:fconvex} and \ref{subsec:fconcave}, respectively).

\section{The Associated Family of Optimal Selling Problems}
\label{sec:OS}

In this section we solve the parameter-dependent optimal stopping problem with regime switching \eqref{def-u}. This result is of interest on its own since problem \eqref{def-u} takes the form of an optimal selling problem in a Bachelier model with regime switching, and with a transaction cost $\theta(y)$ that parametrically depends on $y\in(0,1]$. In the rest of this section $y \in (0,1]$ is given and fixed.

Some preliminary properties of $u$ are stated in the next proposition, whose proof can be found in Appendix \ref{someproofs}. These properties of $u$ will be important in the following when constructing the solution to \eqref{def-u}. 

\begin{proposition}
\label{preliminary:OS}
Recall \eqref{def-u}. There exists a constant $K(y)>0$ such that for any $(x,i) \in \mathbb{R} \times \{1,2\}$
\begin{enumerate}
	\item $u(x,i;y) \geq x-\theta(y)$;
	\item $|u(x,i;y)| \leq K(y)(1+ |x|)$.
\end{enumerate}
\end{proposition}

In line with the standard theory of optimal stopping (see, e.g., \cite{PeskShir}) we expect $u$ of \eqref{def-u} to suitably satisfy the variational inequality 
\beq
\label{HJB-OS-u}
\max\Big\{\big(\mathcal{G} - \rho\big) w(x,i;y), x-\theta(y) - w(x,i;y)\Big\} =0, \qquad (x,i) \in \mathbb{R} \times \{1,2\},
\eeq
for any given $y\in (0,1]$, and where $\mathcal{G}$ has been defined in \eqref{generator}. Also, we define the continuation and stopping regions of \eqref{def-u} as
\begin{equation*}
\label{cont-stop}
\mathcal{C}:=\Big\{(x,i) \in \mathbb{R} \times \{1,2\}: u(x,i;y) > x-\theta(y)\Big\}, \quad \mathcal{S}:=\Big\{(x,i) \in \mathbb{R} \times \{1,2\}: u(x,i;y) = x-\theta(y)\Big\},
\end{equation*}
respectively. Given the structure of optimal stopping problem \eqref{def-u} we expect that 
\begin{equation}
\label{cont-b}
\mathcal{C}:=\Big\{(x,1): x < x^*_1(y)\Big\} \cup \Big\{(x,2): x < x^*_2(y)\Big\},
\end{equation}
for some thresholds, $x^*_i(y)$, $i=1,2$, such that $x^*_i(y) \geq \theta(y)$, $i=1,2$, and depending parametrically on $y\in (0,1]$.

According to this conjecture three configurations are possible: (A) $x^*_1(y) < x^*_2(y)$, (B) $x^*_1(y) = x^*_2(y)$, and (C) $x^*_1(y) > x^*_2(y)$. We now solve \eqref{HJB-OS-u} in cases (A) and (B). Case (C) is completely symmetric to case (A), and it can be treated with similar arguments. We therefore omit its discussion in this paper in the interest of length. In a second step, by a verification argument, we will show that the solution $w$ to \eqref{HJB-OS-u} satisfies $w\equiv u$. As a byproduct we will also provide the optimal stopping rule $\tau^*$. 

\subsection{Case (A):\, $x^*_1(y) < x^*_2(y)$}

Given our conjecture on the structure of continuation and stopping regions, we rewrite \eqref{HJB-OS-u} in the form of a free-boundary problem. That is, we aim at finding $(w(x,1;y),w(x,2;y), x^*_1(y), x^*_2(y))$ that satisfy the following relations:
\begin{align}
\label{FBP-1-a}
\left\{
\begin{array}{ll}
\tfrac{1}{2}\sigma_i^2 w_{xx}(x,i;y)-\rho w(x,i;y) + \lambda_i(w(x,3-i;y)-w(x,i;y)) =0 & \text{for $x<x^*_1(y)$ and $i=1,2$}\\[+4pt]
\tfrac{1}{2}\sigma_2^2 w_{xx}(x,2;y)-\rho w(x,2;y) + \lambda_2(w(x,1;y)-w(x,2;y)) =0 & \text{for $x^*_1(y)<x<x^*_2(y)$},
\end{array}
\right.
\end{align}
\begin{align}
\label{FBP-1-b}
\left\{
\begin{array}{ll}
w(x,1;y) = x-\theta(y) & \text{for $x^*_1(y)\leq x \leq x^*_2(y)$}\\[+4pt]
w(x,1;y) = x-\theta(y) = w(x,2;y) & \text{for $x \geq x^*_2(y)$}.
\end{array}
\right.
\end{align}
Moreover, from \eqref{HJB-OS-u} $w(\cdot,1;y)$ and $w(\cdot,2;y)$ should also satisfy
\begin{align}
\label{FBP-1-c}
\left\{
\begin{array}{ll}
\tfrac{1}{2}\sigma_i^2 w_{xx}(x,i;y)-\rho w(x,i;y) + \lambda_i(w(x,3-i;y)-w(x,i;y)) \leq 0 & \text{for a.e.\ $x \in \mathbb{R}$ and $i=1,2$}\\[+4pt]
w(x,i;y) \geq x -\theta(y), & \text{for $x \in \mathbb{R}$ and $i=1,2$.}
\end{array}
\right.
\end{align}

Recalling that $\sigma_i>0$ and $\lambda_i>0$, $i=1,2$, let $\alpha_1 < \alpha_2 < 0 < \alpha_3 < \alpha_4$ be the roots of the fourth-order equation $\Phi_1(\alpha)\Phi_2(\alpha) - \lambda_1\lambda_2=0$ (see Lemma \ref{4thordereq} in Appendix \ref{app}), with 
\begin{equation}
\label{def:Phiequal}
\Phi_i(\alpha):=-\frac{1}{2}\sigma_i^2\alpha^2+\rho+\lambda_i, \quad i=1,2.
\end{equation}
Then notice that the first equation of \eqref{FBP-1-a} is actually a system of two second-order ordinary differential equations (ODEs). Hence, transforming such a system into a system of four first-order ODEs, one finds that its general solution is given by 
\begin{align}
\label{eq:jointcont}
\left\{
\begin{array}{ll}
w(x,1;y)=A_1(y)e^{\alpha_1 x}+A_2(y)e^{\alpha_2 x}+A_3(y)e^{\alpha_3 x}+A_4(y)e^{\alpha_4 x}\\[+4pt]
w(x,2;y)=B_1(y)e^{\alpha_1 x}+B_2(y)e^{\alpha_2 x}+B_3(y)e^{\alpha_3 x}+B_4(y)e^{\alpha_4 x},
\end{array}
\right.
\end{align}
for any $x<x^*_1(y)$, $x^*_1(y)$ to be found, and where $B_j(y) :=\frac{\Phi_1(\alpha_j)}{\lambda_1}A_j(y) = \frac{\lambda_2}{\Phi_2(\alpha_j)}A_j(y)$, $j=1,2,3,4$, with $A_j(y)$ to be determined. Since the value function \eqref{def-u} diverges at most linearly (cf.\ Proposition \ref{preliminary:OS}) we set $A_1(y)=0=A_2(y)$, so that also $B_1(y)=0=B_2(y)$.
\smallskip

On the other hand, the solution to the second equation of \eqref{FBP-1-a} and the first equation of \eqref{FBP-1-b} is given on $(x^*_1(y),x^*_2(y))$ by
\begin{align}
\label{eq:cont-stop}
\left\{
\begin{array}{ll}
w(x,1;y)=x-\theta(y)\\[+4pt]
w(x,2;y)=B_5(y) e^{\alpha_5 x} + B_6(y) e^{-\alpha_5 x}+\lambda_2\left(\frac{x-\theta(y)}{\rho+\lambda_2}\right),
\end{array}
\right.
\end{align}
with $\alpha_5=\sqrt{\frac{2(\rho+\lambda_2)}{\sigma_2^2}}$, and for some $B_5(y)$ and $B_6(y)$ to be found.
\smallskip

Finally, for any $x\geq x^*_2(y)$ we have (cf.\ the second equation of \eqref{FBP-1-b})
\begin{equation}
\label{eq:stopstop}
w(x,1;y) = x - \theta(y) = w(x,2;y).
\end{equation}

It now remains to find the constants $A_3(y), A_4(y), B_5(y), B_6(y)$ and the two threshold values $x^*_1(y), x^*_2(y)$. To accomplish that we impose that $w(\cdot,1;y)$ is continuous with continuous first order derivative at $x=x^*_1(y)$, and that $w(\cdot,2;y)$ is continuous with continuous first order derivative at $x=x^*_1(y)$ and $x=x^*_2(y)$. In the optimal stopping literature these regularity requirements are the so-called \emph{continuous-fit} ($C^0$-regularity) and \emph{smooth-fit} ($C^1$-regularity) conditions. Then we find from \eqref{eq:jointcont}--\eqref{eq:stopstop} the nonlinear system
\begin{align}
\label{system}
\left\{
\begin{array}{ll}
A_3(y) e^{\alpha_3 x^{*}_1(y)}+A_4(y)e^{\alpha_4 x^*_1(y)}=x^*_1(y)-\theta(y)\\[+5pt]
\alpha_3 A_3(y) e^{\alpha_3 x^{*}_1(y)}+\alpha_4 A_4(y) e^{\alpha_4 x^*_1(y)}=1 \\[+5pt]
 B_3(y)e^{\alpha_3 x^*_1(y)}+B_4(y)e^{\alpha_4 x^*_1(y)} = B_5(y)e^{\alpha_5 x^*_1(y)}+B_6(y)e^{-\alpha_5 x^*_1(y)}+\lambda_2\left(\frac{x^*_1(y)-\theta(y)}{\rho+\lambda_2}\right) \\[+5pt]
\alpha_3 B_3(y)e^{\alpha_3 x^*_1(y)}+\alpha_4 B_4(y)e^{\alpha_4 x^*_1(y)} = \alpha_5 B_5(y) e^{\alpha_5 x^*_1(y)}-\alpha_5 B_6(y) e^{-\alpha_5 x^*_1(y)}+\frac{\lambda_2}{\rho+\lambda_2}\\ [+4pt] B_5(y) e^{\alpha_5 x^*_2(y)}+B_6(y) e^{-\alpha_5 x^*_2(y)}+\lambda_2\left(\frac{x^*_2(y)-\theta(y)}{\rho+\lambda_2}\right)=x^*_2(y)-\theta(y) \\[+5pt]
\alpha_5 B_5(y)e^{\alpha_5 x^*_2(y)}-\alpha_5 B_6(y) e^{-\alpha_5 x^*_2(y)}+\frac{\lambda_2}{\rho+\lambda_2}=1.
\end{array}
\right.
\end{align}

Solving the first two equations of \eqref{system} with respect to $A_3(y)$ and $A_4(y)$ we obtain after some simple algebra
\begin{equation}
\label{A3-A4}
A_3(y) = \Big[\frac{\alpha_4(x^*_1(y)-\theta(y))-1}{\alpha_4-\alpha_3}\Big]e^{-\alpha_3 x^*_1(y)}, \quad A_4(y) = \Big[\frac{1-\alpha_3(x^*_1(y)-\theta(y))}{\alpha_4-\alpha_3}\Big]e^{-\alpha_4 x^*_1(y)}.
\end{equation}
Analogously, the solution to the fifth and the sixth equations of \eqref{system} is given in terms of the unknown $x^*_2(y)$ as
\begin{align}
\label{B5-B6}
\left\{
\begin{array}{ll}
\displaystyle B_5(y)=\frac{\rho}{\rho+\lambda_2}\left[\frac{e^{-\alpha_5 x^*_2(y)}(1+\alpha_5(x^*_2(y)-\theta(y)))}{2\alpha_5}\right]\\[+4pt]
\\ [+4pt]
\displaystyle B_6(y)=\frac{\rho}{\rho+\lambda_2}\left[\frac{e^{\alpha_5 x^*_2(y)}(\alpha_5(x^*_2(y)-\theta(y))-1)}{2\alpha_5}\right].
\end{array}
\right.
\end{align}
Finally, plugging \eqref{A3-A4} and \eqref{B5-B6} into the third and the fourth equations of \eqref{system}, recalling that $B_3(y)=\frac{\Phi_1(\alpha_3)}{\lambda_1}A_3(y)$ and $B_4(y)=\frac{\Phi_1(\alpha_4)}{\lambda_1}A_4(y)$, we find after some algebra that $(x^*_1(y),x^*_2(y))$ should satisfy
\beq
\label{system-boundaries}
F_1(x^*_1(y),x^*_2(y);y)=0 \qquad \mbox{and} \qquad F_2(x^*_1(y),x^*_2(y);y)=0,
\eeq
where we have set
\begin{align}
\label{F1F2}
\left\{
\begin{array}{ll}
F_1(u,v;y):=\frac{\rho}{\rho+\lambda_2}\Big[(v-\theta(y))\cosh\Big(\alpha_5(v-u)\Big)-\frac{1}{\alpha_5}\sinh\Big(\alpha_5(v-u)\Big)\Big]+a_1(u-\theta(y))+a_2\\[+6pt]
\\ [+2pt]
F_2(u,v;y):=\frac{\rho}{\rho+\lambda_2}\Big[\cosh\Big(\alpha_5(v-u)\Big)-{\alpha_5}(v-\theta(y))\sinh\Big(\alpha_5(v-u)\Big)\Big]+a_3(u-\theta(y))+a_4,
\end{array}
\right.
\end{align}
with $a_i:=a_i(\rho, \lambda_1,\lambda_2,\sigma_1, \sigma_2)$, $i=1,2,3,4$, given by
\begin{align}
\label{a1-a4}
\left\{
\begin{array}{ll}
\displaystyle a_1:=-\frac{\alpha_4\Phi_1(\alpha_3)-\alpha_3\Phi_1(\alpha_4)}{\lambda_1(\alpha_4-\alpha_3)}+\frac{\lambda_2}{\rho+\lambda_2}, & \quad \displaystyle a_2:=\frac{\Phi_1(\alpha_3)-\Phi_1(\alpha_4)}{\lambda_1(\alpha_4-\alpha_3)},\\[+5pt]
\displaystyle a_3:=\frac{\alpha_3\alpha_4}{\lambda_1(\alpha_4-\alpha_3)}[\Phi_1(\alpha_4)-\Phi_1(\alpha_3)],& \quad
\displaystyle a_4:=\frac{\alpha_3\Phi_1(\alpha_3)-\alpha_4\Phi_1(\alpha_4)}{\lambda_1(\alpha_4-\alpha_3)} +\frac{\lambda_2}{\rho+\lambda_2}.
\end{array}
\right.
\end{align}
Notice that $a_1 <0$, $a_2>0$, $a_3 <0$ and $a_4>0$ by Lemma \ref{lemm:valuesai} in Appendix \ref{app}. 

Since we expect from \eqref{def-u} that $x^*_i(y)$, $i=1,2$, are such that $x^*_2(y) > x^*_1(y) \geq \theta(y)$, it is natural to check if \eqref{system-boundaries} admits a solution in $(\theta(y), \infty) \times (\theta(y), \infty)$.
So far we do not know about existence, and in case uniqueness, of such a solution. To investigate this fact we define $z^*_1(y):=x_1^*(y)-\theta(y)$ and $z^*_2(y):=x_2^*(y)-x_1^*(y)$, so that $x_2^*(y)-\theta(y)= z^*_1(y) + z^*_2(y)$, and we notice that with such a definition the explicit dependence with respect to $y$ disappears in \eqref{system-boundaries}. We can thus drop the $y$-dependence in $z^*_i(y)$, $i=1,2$, and set $(z^*_1,z^*_2)$ as the solution, if it does exist, of the equivalent system
\beq
\label{system-boundaries-2}
G_1(u,v)=0 \qquad \mbox{and} \qquad G_2(u,v)=0,
\eeq
with
\begin{align}
\label{G1G2}
\left\{
\begin{array}{ll}
G_1(u,v):=(a_1+\frac{\rho}{\rho+\lambda_2}\cosh(\alpha_5 v))u-\frac{\rho}{\rho+\lambda_2}[\frac{1}{\alpha_5}\sinh(\alpha_5 v)-v\cosh(\alpha_5 v)]+a_2\\[+4pt]
\\ [+2pt]
G_2(u,v):=(a_3-\frac{\rho\alpha_5}{\rho+\lambda_2}\sinh(\alpha_5v))u - \frac{\rho}{\rho+\lambda_2}[v\alpha_5\sinh(\alpha_5v)-\cosh(\alpha_5v)]+a_4,
\end{array}
\right.
\end{align}
for $u,v\geq 0$.

\begin{proposition}
\label{prop:existence}
Let $\hat{z}_2$ be the unique positive solution to the equation
$$a_1+\frac{\rho}{\rho+\lambda_2}\cosh(\alpha_5 v) =0, \quad v \geq 0,$$
with $a_1$ as in \eqref{a1-a4} and $\alpha_5=\sqrt{\frac{2(\rho+\lambda_2)}{\sigma_2^2}}$.
Then there exists a unique couple $(z^*_1,z^*_2)$ solving \eqref{system-boundaries-2} in $(0,\infty) \times (0, \hat{z}_2)$ if and only if $\sigma_1^2 < \sigma_2^2$. Moreover $z^*_1$ is such that 
$$-\frac{a_2}{a_1+\frac{\rho}{\rho+\lambda_2}} < z^*_1 < -\frac{\frac{\rho}{\rho+\lambda_2}+a_4}{a_3}.$$
\end{proposition}

\begin{proof}
\emph{Step 1.} Note that the function $r(v):=\frac{\rho}{\rho+\lambda_2}[\frac{1}{\alpha_5}\sinh(\alpha_5 v)-v\cosh(\alpha_5 v)] - a_2$, $v\geq 0$, is strictly decreasing, and therefore strictly negative for any $v\geq 0$ since $r(0) = - a_2 < 0$ (cf.\ Lemma \ref{lemm:valuesai} in Appendix \ref{app}).
\medskip

\emph{Step 2.} Here we prove that the equation $h(v)=0$ with $h(v):=a_1+\frac{\rho}{\rho+\lambda_2}\cosh(\alpha_5 v)$, $v \geq 0$, admits a unique solution $\hat{z}_2 > 0$. For this it suffices to notice that $v\mapsto h(v)$ is strictly increasing with $\lim_{v\rightarrow \infty}h(v) = +\infty$, and that $h(0)=a_1+\frac{\rho}{\rho+\lambda_2} = - \frac{\rho + \frac{1}{2}\sigma_1^2 \alpha_3\alpha_4}{\lambda_1} < 0$. The last inequality in the previous formula follows by using \eqref{a1final} of Appendix \ref{app}.
\medskip

\emph{Step 3.} By \emph{Step 2} for any $v \in [0,\hat{z}_2)$ we can rewrite \eqref{system-boundaries-2} in the equivalent form 
$$u = M_1(v), \qquad M_1(v) - M_2(v) =0,$$
with
\begin{align}
\label{eq:M1M2}
\left\{
\begin{array}{ll}
\displaystyle M_1(v):=\frac{\frac{\rho}{\rho+\lambda_2}[\frac{1}{\alpha_5}\sinh(\alpha_5v)-v\cosh(\alpha_5 v)]-a_2}{a_1+\frac{\rho}{\rho+\lambda_2}\cosh(\alpha_5 v)}\\[+4pt]
\\ [+2pt]
\displaystyle M_2(v):=\frac{\frac{\rho}{\rho+\lambda_2}[v\alpha_5\sinh(\alpha_5v)-\cosh(\alpha_5v)]-a_4}{a_3-\frac{\rho\alpha_5}{\rho+\lambda_2}\sinh(\alpha_5v)},
\end{array}
\right.
\end{align}
where we have also used the fact that $a_3-\frac{\rho\alpha_5}{\rho+\lambda_2}\sinh(\alpha_5v) \neq 0$ on $[0,\infty)$ being $a_3<0$ (see again Lemma \ref{lemm:valuesai} in Appendix \ref{app}).

The numerator of $M_1$ in \eqref{eq:M1M2} is strictly negative on $v\geq 0$ by \emph{Step 1}. Using this fact, and noticing that $a_1+\frac{\rho}{\rho+\lambda_2}\cosh(\alpha_5 v) <0$ on $[0, \hat{z}_2)$, by direct calculations one can observe that $v \mapsto M_1(v)$ strictly increases on $[0, \hat{z}_2)$, and it is such that $\lim_{z \uparrow \hat{z}_2}M_1(v) = + \infty$.

Also, one can check by employing \eqref{Phis-1} and \eqref{Phis-2} of Lemma \ref{lemm:valuesai}, and the definitions of $\alpha_3$ and $\alpha_4$, that $M_1(0)-M_2(0)=\frac{1}{a_3}\Big(\frac{\rho}{\rho+\lambda_2}+a_4\Big)-\frac{a_2}{a_1+\frac{\rho}{\rho+\lambda_2}} < 0$ if and only if $\sigma^2_1 < \sigma^2_2$. We now claim (and prove later) that $v\mapsto M_2(v)$ strictly decreases in $[0,\hat{z}_2]$, so that $v\mapsto M_1(v)-M_2(v)$ strictly increases on $[0,\hat{z}_2)$ and diverges to $+\infty$ as $z$ approaches $\hat{z}_2$. Combining all these facts we conclude that there exists a unique $z^*_2 \in (0, \hat{z}_2)$ solving $M_1(v) - M_2(v) =0$. Hence, $z^*_1 = M_1(z^*_2)$ (or, equivalently, $z^*_1=M_2(z^*_2)$), and $z^*_1>0$ because $M_1(v) \geq M_1(0) > 0$ on $[0, \hat{z}_2)$. 

Moreover, since $M_1(\cdot)$ is strictly increasing, $M_2(\cdot)$ is strictly decreasing on $[0,\hat{z}_2)$, and $z_2^* < \hat{z}_2$, one has $M_1(0) < z^*_1 < M_2(0)$; i.e.,
\begin{equation}
\label{w1convex-2-a}
0< -\frac{a_2}{a_1+\frac{\rho}{\rho+\lambda_2}} < z^*_1 < -\frac{\frac{\rho}{\rho+\lambda_2}+a_4}{a_3}.
\end{equation}
\medskip

\emph{Step 4.} To complete the proof we need to show that $v\mapsto M_2(v)$ is strictly decreasing in $[0,\hat{z}_2]$. By direct calculations one can see that the latter monotonicity property holds if $$-\frac{\rho}{\rho+\lambda_2}\cosh(\alpha_5v) + a_3 v <0$$ on $[0,\hat{z}_2]$. But this is true since $a_3 < 0$.
\end{proof}


Since by Proposition \ref{prop:existence} there exists a unique couple $(z^*_1,z^*_2)$ solving \eqref{system-boundaries-2} in $(0,\infty) \times (0, \hat{z}_2)$ if and only if $\sigma_1^2 < \sigma_2^2$, the latter condition is taken as a standing assumption throughout the rest of this section.

\begin{corollary}
\label{cor:boundary-system}
There exists a unique couple $(x_1^*(y), x_2^*(y)) \in (\theta(y), + \infty) \times (\theta(y), + \infty)$ solving \eqref{system-boundaries}. Moreover, it is such that $x_2^*(y) > x_1^*(y)$. 
\end{corollary}
\begin{proof}
By Proposition \ref{prop:existence} there exists a unique couple $(z^*_1,z^*_2)$ solving \eqref{system-boundaries-2} in $(0,\infty) \times (0, \hat{z}_2)$. Since $z^*_1=x_1^*(y)-\theta(y)$ and $z^*_2=x_2^*(y)-x_1^*(y)$, one has $x_1^*(y) = z^*_1 + \theta(y) > \theta(y)$ and $x_2^*(y)= z^*_2 + x_1^*(y) > x_1^*(y) > \theta(y)$. 
\end{proof}



Theorem \ref{candidate-w} below proves that $(w(x,1;y), w(x,2;y),x^*_1(y), x^*_2(y))$ solve free-boundary problem \eqref{FBP-1-a}-\eqref{FBP-1-c}. Its proof is quite long and technical, and for this reason it is postponed to Appendix \ref{someproofs}.

\begin{theorem}
\label{candidate-w}
\textbf{[The Candidate Value Function]} Let $(x_1^*(y), x_2^*(y))$ with $x_2^*(y) > x_1^*(y)$ be the unique solution to \eqref{system-boundaries} in $(\theta(y), + \infty) \times (\theta(y), + \infty)$. Define $A_3(y)$ and $A_4(y)$ as in \eqref{A3-A4}, $B_3(y):=\frac{\Phi_1(\alpha_3)}{\lambda_1}A_3(y)$ and $B_4(y):=\frac{\Phi_1(\alpha_4)}{\lambda_1}A_4(y)$, and $B_5(y)$ and $B_6(y)$ as in \eqref{B5-B6}. Then the functions
\begin{align}
\label{eq:candidatevalue1}
w(x,1;y):=
\left\{
\begin{array}{ll}
A_3(y)e^{\alpha_3 x}+A_4(y)e^{\alpha_4 x}, & \quad x \leq x^*_1(y)\\[+5pt]
x - \theta(y), & \quad x \geq x^*_1(y),
\end{array}
\right.
\end{align}
and 
\begin{align}
\label{eq:candidatevalue2}
w(x,2;y):=
\left\{
\begin{array}{ll}
B_3(y)e^{\alpha_3 x}+B_4(y)e^{\alpha_4 x}, & \quad x \leq x^*_1(y)\\[+5pt]
B_5(y) e^{\alpha_5 x} + B_6(y) e^{-\alpha_5 x}+\lambda_2\left(\frac{x-\theta(y)}{\rho+\lambda_2}\right), & \quad x^*_1(y) \leq x \leq x^*_2(y)\\[+5pt]
x - \theta(y), & \quad x \geq x^*_2(y),
\end{array}
\right.
\end{align}
are such that $w(\cdot,i;y) \in C^1(\mathbb{R})$ with $w_{xx}(\cdot,i;y) \in L^{\infty}_{loc}(\mathbb{R})$ for any $i=1,2$, and $|w(x,i;y)|\leq \kappa_i(y)(1 + |x|)$ for some $\kappa_i(y) > 0$. Moreover, $(w(x,1;y), w(x,2;y),x^*_1(y), x^*_2(y))$ solve free-boundary problem \eqref{FBP-1-a}-\eqref{FBP-1-c}.
\end{theorem}

We now verify the actual optimality of the candidate value function of Theorem \ref{candidate-w}. The proof of this result is contained in Appendix \ref{someproofs}.

\begin{theorem}
\label{thm:verifying}
\textbf{[The Verification Theorem]} Let $\mathcal{C}=\{(x,1): x < x_1^*(y)\} \cup \{(x,2): x < x_2^*(y)\}$. Then, for $w$ as in Theorem \ref{candidate-w} and for $u$ as in \eqref{def-u}, one has that $w = u$ on $\mathbb{R}\times \{1,2\}$ and
\beq
\label{tau-optimal}
\tau^*:=\inf\{t\geq 0 : (X_{t},\varepsilon_{t})\not \in \mathcal{C}\},\quad \mathbb{P}_{(x,i)}-a.s.,
\eeq
is an optimal stopping time.
\end{theorem}


\subsection{Case (B):\, $x^*_1(y)= x^*_2(y)$}
\label{caseB}

In this section we study the case in which the two boundaries $x^*_1(y)$ and $x^*_2(y)$ coincide and are equal to some $x^*(y)$ to be found. We will find that the value function is regime-independent as well, and equals the value function that one would obtain in a model without regime switching.

We rewrite \eqref{HJB-OS-u} in the form of a free-boundary problem to find $(w(x,1;y),w(x,2;y), x^*(y))$, with $w(\cdot,i;y)\in C^1(\RR)$ and $w_{xx}(\cdot,i;y)\in L^{\infty}_{loc}(\RR)$ for any $i=1,2$, solving
\begin{align}
\label{FBP-equal}
\left\{
\begin{array}{ll}
\tfrac{1}{2}\sigma_i^2 w_{xx}(x,i;y)-\rho w(x,i;y) + \lambda_i(w(x,3-i;y)-w(x,i;y)) =0 & \text{for $x<x^*(y)$ and $i=1,2$}\\[+4pt]
w(x,i;y) = x-\theta(y) & \text{for $x \geq x^*(y)$}\\[+4pt]
\tfrac{1}{2}\sigma_i^2 w_{xx}(x,i;y)-\rho w(x,i;y) + \lambda_i(w(x,3-i;y)-w(x,i;y)) \leq 0 & \text{for a.e.\ $x \in \mathbb{R}$ and $i=1,2$}\\[+4pt]
w(x,i;y) \geq x -\theta(y), & \text{for $x \in \mathbb{R}$ and $i=1,2$}.\\[+4pt]
\end{array}
\right.
\end{align}


Recall \eqref{def:Phiequal} and that $\alpha_1 < \alpha_2 <0 <\alpha_3 < \alpha_4$ denote the solutions to the fourth-order equation $\Phi_1(\alpha)\Phi_2(\alpha) -\lambda_1\lambda_2 =0$ (cf.\ Lemma \ref{4thordereq} in Appendix \ref{app}). Then the general solution to the system of two second-order ODEs appearing in the first line of \eqref{FBP-equal} is given for any $x<x^*(y)$ by
\begin{equation}
\label{valuecontin}
\left\{
\begin{array}{ll}
w(x,1;y)=\widetilde{A}_1(y)e^{\alpha_1x}+\widetilde{A}_2(y)e^{\alpha_2x}+\widetilde{A}_3(y)e^{\alpha_3x}+\widetilde{A}_4(y)e^{\alpha_4x}\\[+4pt]
w(x,2;y)=\widetilde{B}_1(y)e^{\alpha_1x}+\widetilde{B}_2(y)e^{\alpha_2x}+\widetilde{B}_3(y)e^{\alpha_3x}+\widetilde{B}_4(y)e^{\alpha_4x},\\
\end{array}
\right.
\end{equation}
with
\beq
\label{Bjs}
\widetilde{B}_j(y)=\frac{\Phi_1(\alpha_j)}{\lambda_1}\widetilde{A}_j(y)=\frac{\lambda_2}{\Phi_2(\alpha_j)}\widetilde{A}_j(y),\quad j=1,2,3,4.
\eeq
Notice that from the expressions of $\alpha_3$ and $\alpha_4$ (see the proof of Lemma \ref{4thordereq} in Appendix \ref{app}) one has $\Phi_1(\alpha_3)>0$ and $\Phi_1(\alpha_4)<0$.
Since for $x \to -\infty$ the value function diverges at most with linear growth (cf.\ Proposition \ref{preliminary:OS}) we set $\widetilde{A}_1(y)=\widetilde{A}_2(y) = 0 = \widetilde{B}_1(y)=\widetilde{B}_2(y)$.

For $x \in [x^*(y),+\infty)$ we have from \eqref{FBP-equal}
\begin{equation}
\label{valuestopping}
w(x,1;y)=x-\theta(y) = w(x,2;y).
\end{equation}

It now only remains to find $\widetilde{A}_3(y)$, $\widetilde{A}_4(y)$ and $x^*(y)$, since $\widetilde{B}_3(y)$ and $\widetilde{B}_4(y)$ are given in terms of $\widetilde{A}_3(y)$ and $\widetilde{A}_4(y)$ through \eqref{Bjs}.
To do so, we impose that $w(\cdot, i;y)$, $i=1,2$, is continuous across $x^*(y)$ together with its first derivative (i.e.\ continuous-fit and smooth-fit conditions), and we obtain the system
\begin{equation}
\label{system-C1}
\left\{
\begin{array}{ll}
\widetilde{A}_3(y) e^{\alpha_3 x^{*}(y)}+\widetilde{A}_4(y)e^{\alpha_4 x^*(y)}=x^*(y)-\theta(y) \\ \\
\alpha_3 \widetilde{A}_3(y) e^{\alpha_3 x^{*}(y)}+\alpha_4 \widetilde{A}_4(y) e^{\alpha_4 x^*(y)}= 1\\ \\
\widetilde{B}_3(y) e^{\alpha_3 x^{*}(y)}+\widetilde{B}_4(y) e^{\alpha_4 x^*(y)}=x^*(y)-\theta(y) \\ \\
\alpha_3 \widetilde{B}_3(y) e^{\alpha_3 x^{*}(y)}+\alpha_4 \widetilde{B}_4(y) e^{\alpha_4 x^*(y)}=1.
\end{array}
\right.
\end{equation}

Solving the first two equations of \eqref{system-C1} for $\widetilde{A}_3(y)$ and $\widetilde{A}_4(y)$, one has
\beq
\label{A3A4equal}
\widetilde{A}_3(y)=\Big[\frac{\alpha_4(x^*(y)-\theta(y))-1}{(\alpha_4-\alpha_3)}\Big]e^{-\alpha_3 x^*(y)},\qquad \widetilde{A}_4(y)=\Big[\frac{1-\alpha_3(x^*(y)-\theta(y))}{(\alpha_4-\alpha_3)}\Big]e^{-\alpha_4 x^*(y)}.
\eeq
On the other hand, recalling \eqref{Bjs} and plugging $\widetilde{A}_3(y)$ and $\widetilde{A}_4(y)$ from \eqref{A3A4equal} into the third equation of \eqref{system-C1}, some simple algebra leads to
\beq
\label{xstar}
x^*(y)=\frac{\frac{1}{2}\sigma_1^2(\alpha_3+\alpha_4)}{\rho+\frac{1}{2}\sigma_1^2\alpha_3\alpha_4}+\theta(y),
\eeq
where \eqref{def:Phiequal} has also been used.

Similarly, inserting $\widetilde{A}_3(y)$ and $\widetilde{A}_4(y)$ from \eqref{A3A4equal} into the fourth equation of \eqref{system-C1} and using \eqref{def:Phiequal} one obtains
\beq
\label{xstar2}
x^*(y)=\frac{\frac{1}{2}\sigma_1^2(\alpha_3^2+\alpha_4^2+\alpha_3\alpha_4) - \rho}
{\frac{1}{2}\sigma_1^2\alpha_3\alpha_4(\alpha_3+\alpha_4)}+\theta(y).
\eeq
Equations \eqref{xstar} and \eqref{xstar2} then imply that system \eqref{system-C1} admits a solution (which is then unique) if and only if
\beq
\label{consistency}
\frac{\frac{1}{2}\sigma_1^2(\alpha_3+\alpha_4)}{\rho+\frac{1}{2}\sigma_1^2\alpha_3\alpha_4}
=\frac{\frac{1}{2}\sigma_1^2(\alpha^2_3+\alpha^2_4+\alpha_3\alpha_4) - \rho}
{\frac{1}{2}\sigma_1^2\alpha_3\alpha_4(\alpha_3+\alpha_4)}.
\eeq

Using that $(\alpha_3\alpha_4)^2 = 4[(\rho + \lambda_1)(\rho + \lambda_2) - \lambda_1\lambda_2]/\sigma_1^2\sigma_2^2$, and that $\alpha^2_3 + \alpha^2_4 = 2\sigma_1^2(\rho + \lambda_2) + 2\sigma_2^2(\rho + \lambda_1)/\sigma_1^2\sigma_2^2$ by Vieta's formulas, one can show that \eqref{consistency} is equivalent to $\sigma^2_1 = \sigma^2_2=:\sigma^2$. In such a case, it is not hard to check by direct calculations that $\alpha^2_3=2\rho/\sigma^2$ and $\alpha^2_4=2(\rho+\lambda_1+\lambda_2)/\sigma^2$. Then employing \eqref{Bjs} this in turn gives 
\beq
\label{Bjs-bis}
\widetilde{B}_3(y) = \widetilde{A}_3(y)=\frac{\sigma}{\sqrt{2\rho}}e^{-\frac{\sqrt{2\rho}}{\sigma}x^*(y)} \qquad \text{and} \qquad \widetilde{B}_4(y) = -\frac{\lambda_2}{\lambda_1}\widetilde{A}_4(y) =0.
\eeq
Moreover,
\beq
\label{bd-equal}
x^*(y) = \frac{\sigma}{\sqrt{2\rho}} + \theta(y) > \theta(y).
\eeq
Combining all the previous results, we find that for any $i=1,2$ the candidate value function is
\begin{align}
\label{eq:candidatevalue1-equal}
w(x,i;y):=
\left\{
\begin{array}{ll}
\frac{\sigma}{\sqrt{2\rho}}e^{\frac{\sqrt{2\rho}}{\sigma}(x-x^*(y))}, & \quad x \leq x^*(y),\\[+5pt]
x - \theta(y), & \quad x \geq x^*(y).
\end{array}
\right.
\end{align}

It is easily verified that $(x^*,w)$ as in \eqref{bd-equal} and \eqref{eq:candidatevalue1-equal} equal the free boundary and the value function that we would obtain in a model without regime-switching. Also, by direct calculations one can show that \eqref{bd-equal} and \eqref{eq:candidatevalue1-equal} solve \eqref{FBP-equal}. In particular, $(x^*,w)$ solve the first two lines in \eqref{FBP-equal} by construction, and they fulfill the third equation in \eqref{FBP-equal} because $x^*(y)>\theta(y)$. On the other hand, the fourth equation in \eqref{FBP-equal} follows by the convexity of $w(\cdot,i;y)$ and the fact that $w_x(x^*(y),i;y)=1$ by construction. Then by a standard verification theorem (which is left to the reader in the interest of length) one obtains the next result.
\begin{theorem}
\label{candidate-w-equal}
Assume $\sigma_1=\sigma_2$, let $x^*(y)$ be given by \eqref{bd-equal}, and $w$ as in \eqref{eq:candidatevalue1-equal}. Then the value function of \eqref{def-u} is such that $u\equiv w$.
Moreover, letting $\mathcal{C}=\{(x,i) \in \mathbb{R} \times \{1,2\}: x < x^*(y)\}$, the stopping time
\beq
\label{tau-optimal-equal}
\tau^*:=\inf\{t\geq 0: (X_{t},\varepsilon_{t})\not \in \mathcal{C}\},\quad \mathbb{P}_{(x,i)}-a.s.,
\eeq
is optimal.
\end{theorem}


\section{The Optimal Extraction Policy}
\label{sec:solSSC}

In this section we provide the solution to the finite-fuel singular stochastic control problem \eqref{value} in terms of the solution to the optimal stopping problem with regime switching \eqref{def-u}. In particular, we consider separately the two cases (I) $y \mapsto f(y)$ strictly convex on $[0,1]$, and (II) $y \mapsto f(y)$ concave on $[0,1]$ (cf.\ Assumption \ref{Asscost}). It turns out that the optimal extraction rule is qualitatively different across these two cases.

\subsection{Case (I): $y \mapsto f(y)$ strictly convex on $[0,1]$}
\label{subsec:fconvex}

Assume that $y \mapsto f(y)$ fulfills condition (I) of Assumption \ref{Asscost}. For any $y \in [0,1]$, let $\theta(y)$ in \eqref{def-u} be such that $$\theta(y):= c - \frac{f'(y)}{\rho},$$
and notice that with such a choice of $\theta$ all the results of Section \ref{sec:OS} remains valid for $y \in [0,1]$.

By Corollary \ref{cor:boundary-system} we know that $x_1^*(y) = z^*_1 + \theta(y)$ and $x_2^*(y)= z^*_2 + x_1^*(y)$ (see also \eqref{bd-equal} in the case $x^*_1(y)=x^*_2(y)=x^*(y)$). Because $y \mapsto f(y)$ is continuously differentiable and strictly convex on $[0,1]$, it follows that for any $i=1,2$, $y \mapsto x^*_i(y)$ is continuous and strictly decreasing on $[0,1]$, and it has an inverse with respect to $y$. For $i=1,2$, we then define
\begin{align}\label{def-bstar}
b^*_i(x):=
\left\{
\begin{array}{ll}
1, & x\le x^*_i(1)\\[+2pt]
 (x^*_i)^{-1}(x), & x\in(x^*_i(1),x^*_i(0))\\[+2pt]
0, & x\ge x^*_i(0),
\end{array}
\right.
\end{align}
and we observe that $b^*_i: \mathbb{R} \to [0,1]$ is continuous and decreasing (notice that also the case in which $x^*_1(y) = x^*_2(y)$ - i.e.\ case (B) of Section \ref{caseB} - can be accommodated into \eqref{def-bstar}. Indeed, in such case we simply have $b^*_1=b^*_2$).

We now provide a candidate value function for problem \eqref{value}. To this end, for $u$ as in Theorems \ref{thm:verifying} or \ref{candidate-w-equal}, we introduce the function
\beq
\label{def-U}
F(x,y,i) := \int_0^y u(x, i;z)dz - \frac{f(y)}{\rho}.
\eeq

\begin{proposition}
\label{listprops}
The function $F$ of \eqref{def-U} is such that $F(\cdot, \cdot, i) \in C^{2,1}(\mathbb{R}\times[0,1])$ for any $i=1,2$. Moreover, for $i=1,2$ there exist constants $C_i>0$ and $\kappa_i>0$ such that 
\begin{align}
\label{boundsF}
\big|F(x,y,i)\big| + \big|F_y(x,y,i)\big|\le C_i(1+|x|),\quad \big|F_x(x,y,i)\big| + \big|F_{xx}(x,y,i)\big|\le \kappa_i,
\end{align}
for $(x,y)\in \mathbb{R}\times [0,1]$.
\end{proposition}
\begin{proof}
It is easy to verify from \eqref{eq:candidatevalue1} and \eqref{eq:candidatevalue2}, and from \eqref{eq:candidatevalue1-equal} (upon recalling also Theorems \ref{thm:verifying} and \ref{candidate-w-equal}) that $u$ is of the form $u(x,i;y)=\zeta_i(y)G_i(x)+\eta_i(y)H_i(x)$ for some continuous functions $\zeta_i$, $\eta_i$, $G_i$ and $H_i$. It thus follows that $(x,y)\mapsto F(x,y,i)$  and $(x,y) \mapsto F_y(x,y,i)$ are continuous on $\mathbb{R}\times[0,1]$.
Also, from \eqref{eq:candidatevalue1} and \eqref{eq:candidatevalue2}, and from \eqref{eq:candidatevalue1-equal}, one can see that for any $x$ in a bounded set $\mathcal{K}\subset\mathbb{R}$ and for any $i=1,2$ the derivatives $|u_x|$ and $|u_{xx}|$ are at least bounded by a function $F_{\mathcal{K}}(y)\in L^1(0,1)$. It follows that to determine $F_x$ and $F_{xx}$ one can invoke the dominate convergence theorem and evaluate derivatives inside the integral in \eqref{def-U} so to obtain
\begin{align}
\label{Ux}
F_x(x,y,i) = \int_{0}^{b^*_1(x) \wedge y} u_x(x,i;z)dz + \int_{b^*_1(x) \wedge y}^{b^*_2(x) \wedge y}u_x(x,i;z)dz + \int_{b^*_2(x) \wedge y}^{y}u_x(x,i;z)dz
\end{align}
and
\begin{align}
\label{Uxx}
F_{xx}(x,y,i) = \int_{0}^{b^*_1(x) \wedge y} u_{xx}(x,i;z)dz + \int_{b^*_1(x) \wedge y}^{b^*_2(x) \wedge y}u_{xx}(x,i;z)dz,
\end{align}
where the second integrals on the right hand side of \eqref{Ux} and \eqref{Uxx} equal zero in case $b^*_1=b^*_2$.
Therefore $F(\cdot,\cdot,i) \in C^{2,1}(\mathbb{R}\times[0,1])$ for $i=1,2$ by \eqref{eq:candidatevalue1} and \eqref{eq:candidatevalue2}, \eqref{eq:candidatevalue1-equal}, Theorems \ref{thm:verifying} and \ref{candidate-w-equal}, and continuity of $b^*_i(\cdot)$ (cf.\ \eqref{def-bstar}).
Finally, bounds \eqref{boundsF} follow from \eqref{eq:candidatevalue1} and \eqref{eq:candidatevalue2}, \eqref{eq:candidatevalue1-equal}, \eqref{def-U}, \eqref{Ux} and \eqref{Uxx}.
\end{proof}

The next result shows that $F$ solves the HJB equation \eqref{HJB}.

\begin{proposition}
\label{prop6}
For all $(x,y,i) \in \mathbb{R} \times (0,1] \times \{1,2\}$, $F$ is a classical solution to \eqref{HJB}. Moreover, it satisfies the boundary condition $F(x,0,i)=0$ for $(x,i) \in \mathbb{R} \times \{1,2\}$.
\end{proposition}
\begin{proof}
First of all we observe that for any $(x,y,i) \in \mathcal{O}$ one has by \eqref{def-U} that 
\beq
\label{observation1}
F_y(x,y,i) = u(x,i;y) - \frac{f'(y)}{\rho} \geq x-c,
\eeq 
where the last inequality follows from the fact that $u(x,i;y) \geq x-\theta(y) = x - c + \frac{f'(y)}{\rho}$. In particular, for any $i=1,2$ one has equality in \eqref{observation1} on $\{(x,y) \in \mathbb{R}\times [0,1]: x \geq x^*_i(y)\}$. 

For any fixed $i=1,2$, take $y\in[0,1]$ and $x\in\mathbb{R}$ such that $F_y(x,y,i)>x-c$, i.e.~$y<b^*_i(x)$, and notice that thanks to Proposition \ref{listprops} one can write
$$(\mathcal{G}- \rho)F(x,y,1)= \int_0^y (\mathcal{G}- \rho)u(x,1;z)dz + f(y) =  f(y),$$
and
$$(\mathcal{G}- \rho)F(x,y,2)= \int_0^{y\wedge b^*_1(x)} (\mathcal{G}- \rho)u(x,2;z)dz + \int_{y \wedge b^*_1(x)}^{y\wedge b^*_2(x)} (\mathcal{G}- \rho)u(x,2;z)dz + f(y) =  f(y).$$
The last equalities in the two equations above follow from the fact that $u$ solves free-boundary problem \eqref{FBP-1-a}-\eqref{FBP-1-c} (cf.\ Theorems \ref{candidate-w} and \ref{thm:verifying}; see also Theorem \ref{candidate-w-equal} in the case $x^*_1(y)=x^*_2(y)=x^*(y)$). 

On the other hand, for arbitrary $(x,y,i)\in \mathcal{O}$ we notice that (cf.\ \eqref{listprops})
\begin{eqnarray*}
&&(\mathcal{G}- \rho)F(x,y,i)= \int_{0}^{b^*_1(x) \wedge y} (\mathcal{G}- \rho)u(x,i;z) dz + \int_{b^*_1(x) \wedge y}^{b^*_2(x) \wedge y}(\mathcal{G}- \rho)u(x,i;z)dz \nonumber \\
&& \hspace{0.5cm} + \int_{b^*_2(x) \wedge y}^{y}(\mathcal{G}- \rho)u(x,i;z) dz + f(y) \leq f(y),
\end{eqnarray*}
since, again, $u$ solves free-boundary problem \eqref{FBP-1-a}-\eqref{FBP-1-c}. Hence $F$ solves \eqref{HJB}. Moreover, recalling that $f(0)=0$, it is straightforward to see from \eqref{def-U} that $F(x,0,i) = 0$ for any $(x,i) \in \mathbb{R} \times \{1,2\}$.
\end{proof}

Satisfying \eqref{HJB} and the boundary condition $F(x,0,i)=0$ for $(x,i) \in \mathbb{R} \times \{1,2\}$, $F$ is clearly a candidate value function for problem \eqref{value}. We now introduce a candidate optimal control process. Let $(x,y,i) \in \mathcal{O}$, recall $b^*_i$ of \eqref{def-bstar} and consider the process
\beq
\label{candidateoptimalcontrol}
\nu_t^*=\Big[y - \inf_{0 \leq s < t}b^*_{\varepsilon_s}\big(X_s\big)\Big]^+, \qquad t>0, \quad
\nu_0^*= 0,
\eeq
where $[\,\cdot\,]^+$ denotes the positive part.

\begin{proposition}
\label{admissiblenustar}
The process $\nu^*$ of \eqref{candidateoptimalcontrol} is an admissible control.
\end{proposition}
\begin{proof}
Recall \eqref{admissiblecontrols}. For any given and fixed $\omega\in\Omega$, $t\mapsto \nu^*_t(\omega)$ is clearly nondecreasing and such that $Y^{\nu^*}_t(\omega) \geq 0$, for any $t\geq 0$, since $b^*_i(x) \in [0,1]$ for any $x\in \mathbb{R}$. 
Moreover, since $(X, \varepsilon)$ is right-continuous with left-limits (cf.\ Lemma 3.6 in \cite{ZhuYin}) and $(x,i)\mapsto b^*_i(x)$ is continuous, $t \mapsto \nu^*_t(\omega)$ is left-continuous. Finally, $\mathbb{F}$-progressive measurability of $(X, \varepsilon)$ and measurability of $b^*$ imply that $\nu^*$ is $\mathbb{F}$-progressively measurable by \cite{DM}, Theorem IV.33, whence $\mathbb{F}$-adapted.
\end{proof}


Process $\nu^*$ is the minimal effort needed to have $Y^{\nu^*}_t \leq b^*_{\varepsilon_t}(X_t)$ at any time $t$. In particular it is a standard result (see, e.g., Proposition 2.7 in \cite{DeAFeMo15} and references therein for a proof in a similar setting) that $\nu^*$ of \eqref{candidateoptimalcontrol} solves the Skorokhod reflection problem
\begin{enumerate}
\item $\displaystyle Y^*_t \leq b^*_{\varepsilon_t}(X_t)$, $\mathbb{P}_{(x,y,i)}$-almost surely, for each $t > 0$;
\item $\displaystyle \int_0^T\, \mathds{1}_{\{Y^*_t < b^*_{\varepsilon_t}(X_t)\}}d\nu_t^*=0$ $\mathbb{P}_{(x,y,i)}$-almost surely, for all $T \geq 0$, 
\end{enumerate}
where $Y^*:=Y^{\nu^*}$. An illustration of the (candidate) optimal policy $\nu^*$ is provided in Figure \ref{fig:1}.

\begin{figure}[!ht]
\centering
\includegraphics[scale=0.5]{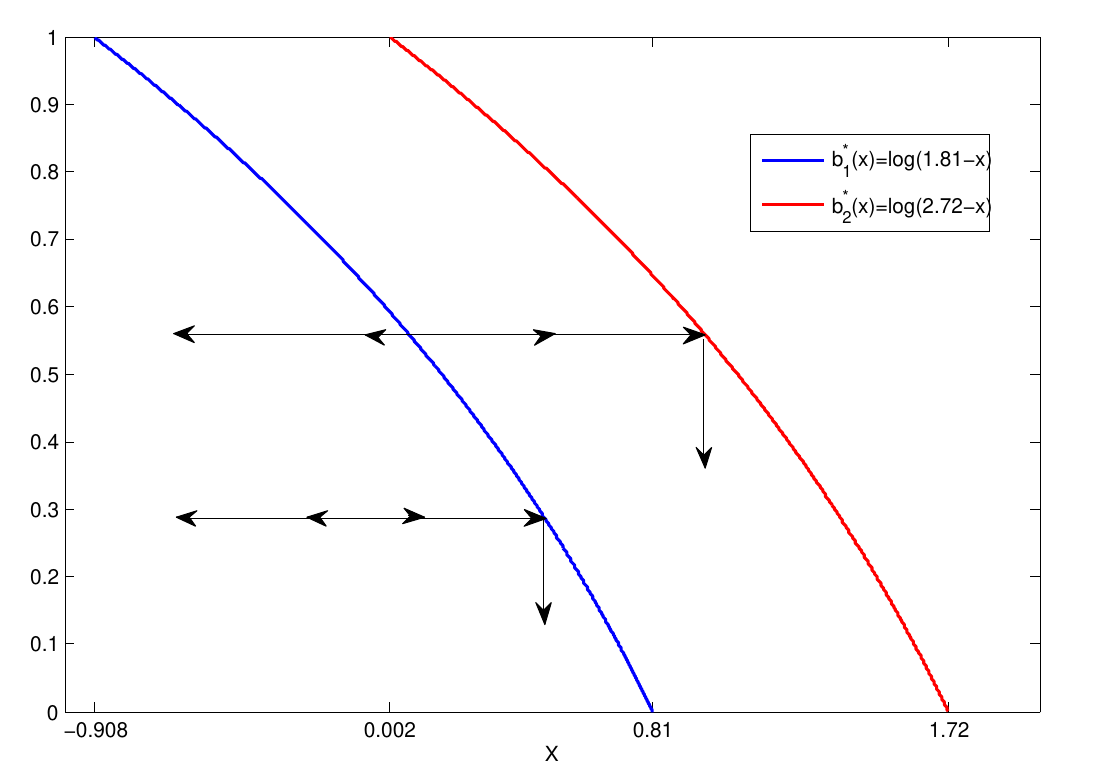}
\caption{\small 
Adopting the terminology of \cite{Guoetal}, the boundaries $b^*_i$, $i=1,2$, split the state space into the \textsl{inaction region} ($y<b^*_1(x)$), \textsl{transient region} ($b^*_1(x)<y<b^*_2(x)$) and \textsl{action region} ($y > b^*_2(x)$). When the initial state is $(x,y,i) \in \mathcal{O}$ with $y<b^*_i(x)$ one observes a Skorokhod reflection of $(X,Y^{*},\varepsilon)$ at $b^*_i$ in the vertical direction up to when all the fuel is spent. If the system is reflected at the upper boundary, at a time of regime switch $\nu^*$ prescribes an immediate jump of $Y^{*}$ from the upper to the lower boundary (whenever they are different). This plot was obtained solving with Matlab the nonlinear system \eqref{system-boundaries-2} when $f(y)=\frac{1}{3}(e^y-1)$ and with $\sigma_1=0.38$, $\sigma_2=1.9$, $\lambda_1=1.7$, $\lambda_2=0.44$, $\rho=1/3$ and $c=0.5$.}
\label{fig:1}
\end{figure}

\begin{theorem}\label{theorem8}
\textbf{[The Verification Theorem]} The control $\nu^*$ of \eqref{candidateoptimalcontrol} is optimal for problem \eqref{value}, and $F$ of \eqref{def-U} is such that $F \equiv V$.
\end{theorem}

\begin{proof}

Since $F$ is a classical solution to the HJB equation due to Proposition \ref{prop6}, one has $F \geq V$ on $\mathcal{O}$ by Theorem \ref{1stverification}. We now show that one actually has $F = V$ on $\mathcal{O}$, and that $\nu^*$ of \eqref{candidateoptimalcontrol} is optimal for problem \eqref{value}.

If $y=0$ then $F(x,0,i) = 0 =V(x,0,i)$. Then take $(x,i) \in \mathbb{R} \times \{1,2\}$, $y\in(0,1]$, set $Y^*:=Y^{\nu^*}$ with $\nu^*$ as in \eqref{admissiblenustar}, and define $\vartheta:=\inf\big\{t\ge 0\,:\,\nu^*_t = y\big\}$ and $\tau_R:=\inf\{t\geq 0: X_t \notin (-R,R)\}$ $\mathbb{P}_{(x,i)}$-a.s., for some $R>0$. Also, let $0 \leq \eta_1 < \eta_2 < ... < \eta_{N} \leq \tau_R \wedge \vartheta$ be the random times of jumps of $\varepsilon$ in the interval $[0,\tau_R \wedge \vartheta)$ (clearly, the number $N$ of those jumps is random as well). Given the regularity of $F$, we can apply It\^o-Meyer's formula for semimartingale (\cite{Meyer}, pp.\ 278-301) to the process $(e^{-\rho t}F(X_{t}, Y^*_{t}, \varepsilon_{t}))_{t\geq 0}$ on each of the intervals $[0, \eta_1)$, $(\eta_1,\eta_2)$,...,$(\eta_N,\tau_R \wedge T)$. Piecing together all the terms we obtain
\begin{align}
\label{Ito-F-1}
F(x,y,i) =& \mathbb{E}_{(x,y,i)}\bigg[e^{-\rho\,(\tau_R\wedge\vartheta)}F(X_{\tau_R\wedge\vartheta}, Y^*_{\tau_R\wedge\vartheta}, \varepsilon_{\tau_R\wedge\vartheta})-\int_0^{\tau_R\wedge\vartheta}e^{-\rho s}(\mathcal{G}-\rho)
F(X_s,Y^*_s,\varepsilon_s)ds\bigg] \nonumber \\
&  + \mathbb{E}_{(x,y,i)}\bigg[\int_0^{\tau_R\wedge\vartheta}e^{-\rho s}F_y(X_s,Y^*_s, \varepsilon_s)d\nu_s^{*,cont}\bigg]\\
& - \mathbb{E}_{(x,y,i)}\left[\sum_{0\le s < \tau_R\wedge\vartheta}e^{-\rho s}
\left(F(X_s,Y^*_{s+}, \varepsilon_s)-F(X_s,Y^*_s,\varepsilon_s)\right)\right]. \nonumber 
\end{align}
Here $\nu^{*,cont}$ denotes the continuous part of $\nu^{*}$.

Recall now \eqref{jump}, that $(\mathcal{G}-\rho)F(x,y,i)=-f(y)$ for $y < b^*_i(x)$ and $F_y(x,y,i) = x - c$ for $y \geq b^*_i(x)$. Furthermore, note that $\nu^*$ solves the Skorokhod reflection problem, and therefore $\{t:\,d\nu^*_t(\omega)>0\} \subseteq \{t:\,Y^*_t(\omega) \geq b^*_{\varepsilon_t(\omega)}(X_t(\omega))\}$ for any $\omega \in \Omega$. Then by using all these facts we obtain from \eqref{Ito-F-1}
\begin{align}
\label{verif07}
F(x,y,i) =& \mathbb{E}_{(x,y,i)}\bigg[e^{-\rho\,(\tau_R\wedge\vartheta)}F(X_{\tau_R\wedge\vartheta}, Y^*_{\tau_R\wedge\vartheta}, \varepsilon_{\tau_R\wedge\vartheta}) - \int_0^{\tau_R\wedge\vartheta}e^{-\rho s}f(Y^*_s)ds  \\
&  +\int_0^{\tau_R\wedge\vartheta}e^{-\rho s}(X_s-c) d\nu^*_s\bigg].\nonumber
\end{align}
As $R\to\infty$, $\tau_R\to\infty$, and clearly $\tau_R\wedge\vartheta\to \vartheta$, $\mathbb{P}_{(x,y,i)}$-a.s. Moreover, we can use the linear growth property of $F$ (cf.~\eqref{boundsF}) and Lemma \ref{lemma:new} in Appendix \ref{app} to apply the dominated convergence theorem and have 
$$\lim_{R \uparrow \infty}\mathbb{E}_{(x,y,i)}\left[e^{-\rho\,(\tau_R\wedge\vartheta)}F(X_{\tau_R\wedge\vartheta}, Y^*_{\tau_R\wedge\vartheta}, \varepsilon_{\tau_R\wedge\vartheta})\right] = \mathbb{E}_{(x,y,i)}\left[e^{-\rho\vartheta}F(X_{\vartheta}, Y^*_{\vartheta}, \varepsilon_{\vartheta})\right] = 0.$$
Finally, we also notice that since $d\,\nu^*_s\equiv0$ and $f(Y^*_s)\equiv0$ for $s>\vartheta$ the integrals in \eqref{verif07} may be extended beyond $\vartheta$ up to $+\infty$ to get
\begin{align}
\label{verif08}
F(x,y,i) =& \mathbb{E}_{(x,y,i)}\bigg[\int_0^{\infty}e^{-\rho s}(X_s-c) d\nu^*_s -\int_0^{\infty}e^{-\rho s}
f(Y^*_s)ds\bigg]=\mathcal{J}_{x,y,i}(\nu^*).
\end{align}
Then $F \equiv V$ and $\nu^*$ is optimal.
\end{proof}

\subsection{Case (II): $y \mapsto f(y)$ concave on $[0,1]$}
\label{subsec:fconcave}

Assume now that $y \mapsto f(y)$ satisfies condition (II) in Assumption \ref{Asscost}, and for $y \in (0,1]$ take $\theta(y)$ in \eqref{def-u} such that $$\theta(y):= c - \frac{1}{\rho}\frac{f(y)}{y}.$$

Recall now $u$ of \eqref{def-u}, and for any $(x,y,i)\in \mathcal{O}$ define the function
\beq
\label{def-W-concave}
W(x,y,i):=y u(x,i;y) - \frac{1}{\rho}f(y).
\eeq
The next result shows that $W$ identifies with a suitable solution to the HJB equation \eqref{HJB}.

\begin{proposition}
\label{prop:Wconcave}
One has that $W(x,0,i)=0$ for all $(x,i) \in \mathbb{R} \times \{1,2\}$, and there exists $K>0$ such that $|W(x,y,i)| \leq K(1 + |x|)$ on $\mathcal{O}$. Moreover, for any $i=1,2$ $W(\cdot, \cdot, i) \in C^0(\mathbb{R} \times [0,1]) \cap C^{1,1}(\mathbb{R} \times (0,1])$ with $W_{xx}(\cdot,\cdot,i) \in L^{\infty}_{loc}(\mathbb{R} \times (0,1])$, and it satisfies the HJB equation \eqref{HJB} in the a.e.\ sense.
\end{proposition}
\begin{proof}
We provide a proof only for $W(x,y,1)$ in the case $x^*_1(y)< x^*_2(y)$, since similar arguments can be employed to deal with all the other cases. 
\medskip

\emph{Step 1.} By Proposition \ref{preliminary:OS} (see in particular the last line in \eqref{boundOS-2}) we can write
\beq
\label{stimaW}
|W(x,y,1)| \leq y |u(x,1;y)| + \frac{1}{\rho}f(y) \leq y\big[2|\theta(y)| + \kappa(1 + |x|)\big] \leq y\big[2c + \kappa(1 + |x|)\big] + \frac{3}{\rho}f(y),
\eeq
for some $\kappa>0$. Taking limit as $y \downarrow 0$, and recalling that $f(0)=0$, we obtain $W(x,0,i)=0$ for all $(x,i) \in \mathbb{R} \times \{1,2\}$. Also, from \eqref{stimaW} we see that the monotonicity of $f(\,\cdot\,)$ and the fact that $y \leq 1$ imply that there exists $K>0$ such that $|W(x,y,i)| \leq K(1 + |x|)$ on $\mathcal{O}$.
\medskip

\emph{Step 2.} As for the claimed regularity of $W(\cdot,\cdot,1)$, one has from \eqref{def-W-concave} that $W \in C^{0,0}(\mathbb{R}\times [0,1])$. Also, from \eqref{eq:candidatevalue1-equal} and Theorem \ref{thm:verifying} it follows that $W_x(\cdot,\cdot,1)$ is uniformly continuous on open sets of the form $(-R,R) \times (\delta,1)$ for $\delta>0$ and arbitrary $R>0$. Hence $W_x(\cdot,\cdot,1)$ has a continuous extension to $\mathbb{R} \times (0,1]$ that we denote again by $W_x(\cdot,\cdot,1)$. Moreover, $W_{xx}(\cdot, \cdot, 1) \in L^{\infty}_{loc}(\mathbb{R} \times (0,1])$. 

We now prove that $W_y(\cdot,\cdot,1) \in C^{0}(\mathbb{R} \times (0,1])$. A direct differentiation of \eqref{def-W-concave}, and the use of \eqref{eq:candidatevalue1-equal} yield for any $y \in [\delta,1]$, $\delta>0$ arbitrary,
\begin{align}
\label{Wy}
&& W_y(x,y,1) = u(x,1;y) + y u_y(x,1;y) - \frac{1}{\rho}f'(y)\nonumber \\
&& = 
\left\{
\begin{array}{ll}
A_3(y)e^{\alpha_3 x}[1 - \alpha_3 y \theta'(y)] + A_4(y)e^{\alpha_4 x}[1 - \alpha_4 y \theta'(y)] - \frac{1}{\rho}f'(y) & \text{for $x<x^*_1(y)$}\\[+4pt]
x-c & \text{for $x>x^*_1(y)$}.\\[+4pt]
\end{array}
\right.
\end{align}
By using \eqref{A3-A4} and exploiting the continuity of $x^*_1(\,\cdot\,)$ (due to continuity of $\theta(\,\cdot\,)$), it can be checked that $y\mapsto W_y(x,y,1)$ is continuous on $[\delta,1]$ for any $x \in \mathbb{R}$. Also, one has that $x\mapsto W_y(x,y,1)$ is continuous on $\mathbb{R}$ uniformly with respect to $y\in [\delta,1]$. In particular, by using once more the expressions for $A_3(y)$ and $A_4(y)$ (cf.\ \eqref{A3-A4}), one has $\lim_{\zeta \downarrow 0}W_y(x^*_1(y)-\zeta,y,1) = x^*_1(y) - c$, uniformly with respect to $y\in [\delta,1]$. Hence $W_y(\cdot,\cdot,1)$ is continuous on $\mathbb{R} \times (0,1]$ by arbitrariness of $\delta>0$.
\medskip

\emph{Step 3.} We here show that $W_y(x,y,1) \geq x - c$ for any $(x,y) \in \mathbb{R} \times (0,1]$. Since this is clearly true on $x>x^*_1(y)$ (cf.\ \eqref{Wy}), we consider only $x<x^*_1(y)$. We show that $W_{yx}(x,y,1) \leq 1$ on $\{(x,y)\in \mathbb{R}\times (0,1]: x < x^*_1(y)\}$, as this fact together with $W_y(x^*_1(y)-,y,1)= x^*_1(y) - c$ implies that $W_y(x,y,1) \geq x - c$ on that set. By differentiating $W_y(x,y,1)$ with respect to $x$ on $\{(x,y)\in \mathbb{R}\times (0,1]: x < x^*_1(y)\}$ one finds that
$$W_{yx}(x,y,1) -1 = u_x(x,1;y) - 1 + y u_{yx}(x,1;y).$$
Theorem \ref{thm:verifying} together with \emph{Step 2} of the proof of Theorem \ref{candidate-w} imply that $u_x(x,1;y) - 1 \leq 0$ for any $x < x^*_1(y)$, $y \in (0,1]$. Moreover, recalling that $x^*_1(y) = z^*_1 + \theta(y)$ (cf.\ Corollary \ref{cor:boundary-system}), it follows from \eqref{eq:candidatevalue1-equal} that $y u_{yx}(x,1;y) = - y \theta'(y) u_{xx}(x,1;y)$ for any $x < x^*_1(y)$ and $y \in (0,1]$. However, by Theorem \ref{thm:verifying} and \emph{Step 2} of the proof of Theorem \ref{candidate-w} we have $u_{xx}(x,1;y) \geq 0$ for $x < x^*_1(y)$, whereas 
\begin{equation}
\label{ythetaprime}
- y \theta'(y) = \frac{1}{\rho}\left[\frac{f'(y)y - f(y)}{y}\right] \leq 0,
\end{equation}
by the assumed concavity of $f$. Hence $W_{yx}(x,y,1) -1 \leq 0$ on $\{(x,y)\in \mathbb{R}\times (0,1]: x < x^*_1(y)\}$, and therefore $W_y(x,y,1) \geq x - c$ on that set.
\medskip

\emph{Step 4.} By Theorems \ref{candidate-w} and \ref{thm:verifying} one has that $(u(x,1;y), u(x,2;y),x^*_1(y), x^*_2(y))$ solve free-boundary problem \eqref{FBP-1-a}-\eqref{FBP-1-c}, and in particular $(\mathcal{G}-\rho)u(x,1;y) \leq 0$ for a.e.\ $x \in \mathbb{R}$ and all $y \in (0,1]$, and with equality for $x < x^*_1(y)$. It thus follows from \eqref{def-W-concave} that $(\mathcal{G}-\rho)W(x,1;y) \leq f(y)$ for a.e.\ $x \in \mathbb{R}$ and for any $y \in (0,1]$, with equality for $x < x^*_1(y)$.
\medskip

Combining the results of the previous steps, the proof is completed.

\end{proof}

Recall that the stopping time
\beq
\label{optimaltimecase3}
\tau^*=\inf\big\{t\ge 0\,:\,X_t \geq x^*_{\varepsilon_t}(y)\big\}, \quad \mathbb{P}_{(x,i)}-a.s.
\eeq
is optimal for \eqref{def-u}, and for any $y\in (0,1]$ define the admissible extraction rule
\begin{align}
\label{op-contr01}
\nu^{\star}_t:=\left\{
\begin{array}{ll}
0, & t\le \tau^*,\\
y, & t>\tau^*.
\end{array}
\right.
\end{align}

\noindent This policy prescribes to instantaneously deplete the reserve at time $\tau^*$. 

\begin{theorem}
\label{thm-opt-c}
The admissible control $\nu^{\star}$ of \eqref{op-contr01} is optimal for problem \eqref{value} and $W\equiv V$.
\end{theorem}
\begin{proof}
Since $W$ solves the HJB equation in the a.e.\ sense due to Proposition \ref{prop6}, one has $W \geq V$ on $\mathcal{O}$ by Theorem \ref{1stverification}. We now show that one actually has $W = V$ on $\mathcal{O}$, and that $\nu^{\star}$ of \eqref{op-contr01} is optimal for problem \eqref{value}.

Let $(x,y,i) \in \mathbb{R} \times (0,1] \times \{1,2\}$, and set $Y^{\star}_t:=Y^{y,\nu^{\star}}_t=y -\nu^{\star}_t$, with $\nu^{\star}$ as in \eqref{op-contr01}. Given the regualrity of $W$, we can apply It\^o-Meyer's formula for semimartingales (cf.\ \cite{Meyer}, pp.\ 278-301) following the approximation argument discussed at the beginning of the proof of Theorem \ref{1stverification}, and then we find that
\begin{align}
\label{Ito}
W(x,y,i) =& \mathbb{E}_{(x,y,i)}\bigg[e^{-\rho\tau^*}W(X_{\tau^*},Y^{\star}_{\tau^*}, \varepsilon_{\tau^*})-\int_0^{\tau^*}e^{-\rho s}f(Y^{\star}_s)ds\bigg] \nonumber \\
&  + \mathbb{E}_{(x,y,i)}\bigg[\int_0^{\tau^*}e^{-\rho s}W_y(X_s,Y^{\star}_s, \varepsilon_s)d\nu_s^{\star,cont}\bigg] \\
& - \mathbb{E}_{(x,y,i)}\bigg[\sum_{0\le s < \tau^*}e^{-\rho s}
\Big(W(X_s,Y^{\star}_{s+}, \varepsilon_s)-W(X_s,Y^{\star}_s,\varepsilon_s)\Big)\bigg] \nonumber \\
=& \mathbb{E}_{(x,y,i)}\bigg[e^{-\rho\tau^*}W(X_{\tau^*},Y^{\star}_{\tau^*}, \varepsilon_{\tau^*})-\int_0^{\tau^*}e^{-\rho s}f(Y^{\star}_s)ds\bigg]. \nonumber
\end{align}
Here $\nu^{\star,cont}$ denotes the continuous part of $\nu^{\star}$. Moreover, we have used that $(\mathcal{G}-\rho)W(X_s,Y^{\star}_s, \varepsilon_s)=f(Y^{\star}_s)$ for any $s \leq \tau^*$, and that the terms in the second and third line of \eqref{Ito} equal zero because $(X_s,Y^{\star}_s,\varepsilon_s)=(X_s,y,\varepsilon_s)$ for $s\le\tau^*$.

On the other hand, \eqref{op-contr01} and the optimality of $\tau^*$ for problem \eqref{def-u} imply that
\begin{align}\label{opt-C01}
\mathbb{E}_{(x,y,i)}&\left[e^{-\rho\tau^*}W(X_{\tau^*},Y^{\star}_{\tau^*}, \varepsilon_{\tau^*})\right]=\mathbb{E}_{(x,y,i)}\left[e^{-\rho\tau^*}W(X_{\tau^*},y, \varepsilon_{\tau^*})\right]\nonumber\\
=& \mathbb{E}_{(x,y,i)}\Big[e^{-\rho\tau^*}\Big(yu(X_{\tau^*},y, \varepsilon_{\tau^*}) - \frac{1}{\rho}f(y)\Big)\Big] = \mathbb{E}_{(x,y,i)}\Big[e^{-\rho\tau^*}\Big(yX_{\tau^*} - y\theta(y)- \frac{1}{\rho}f(y)\Big)\Big] \nonumber\\
=&\mathbb{E}_{(x,y,i)}\left[e^{-\rho\tau^*}(X_{\tau^*} -c )y\right]=\mathbb{E}_{(x,y,i)}\left[\int^\infty_0{e^{-\rho s}(X_s - c) d\nu^{\star}_s}\right].
\end{align}
Also, 
\begin{align}
\label{opt-C02}
\mathbb{E}_{(x,y,i)}\bigg[\int^{\tau^*}_0{e^{-\rho s} f(Y^{\star}_s)ds}\bigg]=\mathbb{E}_{(x,y,i)}\bigg[\int^{\infty}_0 {e^{-\rho s} f(Y^{\star}_s)ds}\bigg],
\end{align}
since $f(Y^{\star}_s) = f(0)$ for any $s>\tau^*$, and $f(0)=0$ by assumption. 

Now, using \eqref{opt-C01} and \eqref{opt-C02} in the last line of \eqref{Ito} gives $W(x,y,i) = \mathcal{J}_{x,y,i}(\nu^{\star}) \leq V(x,y,i)$. Hence, $W = V$ and $\nu^{\star}$ is optimal.
\end{proof}

\begin{remark}
\label{linearcase}
It is worth noticing that the results of this subsection also hold in the case of a running cost function of the form $f(y)=\alpha y$, for some $\alpha \geq 0$. In particular, in such a case $\theta(y)=c - \frac{\alpha}{\rho}$ and does not depend on $y$, so that also the value function $u$ of the auxiliary optimal stopping problem is $y$-independent. It thus follows that $W$ of \eqref{def-W-concave} reads as $W(x,y,i)=yu(x,i) - \frac{\alpha}{\rho}$, and it is immediate to see that it satisfies the HJB equation \eqref{HJB} in the a.e.\ sense.

In fact, when $f(y)=\alpha y$, $\alpha \geq 0$, the optimality of the policy ``instantaneously deplete the reserve as soon as the spot price is sufficiently high" could be expected by noticing that simple algebra and an integration by parts allow to rewrite functional \eqref{functional} as
$$\mathcal{J}_{(x,y,i)}(\nu)=-\frac{\alpha y}{\rho} + \mathbb{E}_{(x,y,i)}\bigg[\int_0^{\infty} e^{-\rho t} \Big(X_t - c + \frac{\alpha}{\rho}\Big)d\nu_t\bigg],\quad (x,y,i) \in \mathcal{O},\,\,\nu\in \mathcal{A}_y,$$
which is linear with respect to the control variable. 
\end{remark}

\begin{remark}
\label{shorttermpayoff}
Although $V(x,0,i)=0$ for $(x,i) \in \mathbb{R}\times (0,1)$, if $\lim_{y \downarrow 0} f'(y) = + \infty$ (Inada condition) one has $V(x,y,i)<0$ for $y$ small enough and for all $x \geq x^*_i(y)$ and $i=1,2$. To see this first of all notice that $x_i^*(y) = const. + \theta(y)$ (see the proof of Corollary \ref{cor:boundary-system}) and the Inada condition yield by de l'H\^opital rule that $\lim_{y \downarrow 0} x^*_i(y) = - \infty$. This is particular implies that for $y$ small enough and for all $x \geq x^*_i(y)$ and $i=1,2$ one has $V(x,y,i) = y\big(x^*_{i}(y) - c \big) < 0$. 
\end{remark}


\section{A Comparison to the No-Regime-Switching Case}
\label{comparison}

It is quite immediate to solve our optimal extraction problem when there is no regime switching. In particular, in this case it can be checked that for any $(0,1]$ the optimal extraction boundary is 

\begin{align}
\label{xcancello}
x^{\scriptstyle{\#}}(y):= \frac{\sigma}{\sqrt{2\rho}} + c + \theta(y) = 
\left\{
\begin{array}{ll}
\frac{\sigma}{\sqrt{2\rho}} + c - \frac{1}{\rho}f'(y) & \text{if $f$ satisfies (I) of Assumption \ref{Asscost},}\\[+4pt]
\frac{\sigma}{\sqrt{2\rho}} + c - \frac{1}{\rho}\frac{f(y)}{y} & \text{if $f$ satisfies (II) of Assumption \ref{Asscost}.}\\[+4pt]
\end{array}
\right.
\end{align}
Consequently, if $f$ satisfies (I) of Assumption \ref{Asscost}, and in particular it is strictly convex on $[0,1]$, the optimal extraction rule reads as
\beq
\label{extraction1reg}
\nu^{\scriptstyle{\#}}_t:=\Big[y - \inf_{0 \leq s < t}b^{\scriptstyle{\#}}\big(X_s\big)\Big]^+, \qquad t>0, \quad
\nu^{\scriptstyle{\#}}_0= 0,
\eeq
where $b^{\scriptstyle{\#}}(\cdot)$ denotes the inverse of $x^{\scriptstyle{\#}}(\cdot)$.
On the other hand, if $f$ satisfies (II) of Assumption \ref{Asscost}, and therefore it is concave on $[0,1]$, it is optimal to extract according to the following policy
\begin{align}
\label{op-contr-noregime-concave}
\nu^{\scriptstyle{\#}}_t:=\left\{
\begin{array}{ll}
0, & t\le \tau^{\scriptstyle{\#}},\\
y, & t>\tau^{\scriptstyle{\#}},
\end{array}
\right.
\end{align}
with $\tau^{\scriptstyle{\#}}:=\inf\big\{t\ge 0\,:\,X_t \geq x^{\scriptstyle{\#}}(y)\big\}$. 

A first observation that is worth making is that $x^{\scriptstyle{\#}} = x^*$, with $x^*$ as in \eqref{bd-equal}. To understand this, recall that in Section \ref{caseB} we have obtained that the two regime-dependent boundaries $x^*_i$, $i=1,2$, coincide and are given by \eqref{bd-equal} if and only if $\sigma_1 = \sigma_2$. In such case the price process does not jump and it therefore behaves as if we had not regime switching. It is then reasonable to obtain for such setting the same optimal selling price that we would obtain in absence of regime shifts.

Although qualitatively similar to \eqref{extraction1reg}, the optimal extraction rule \eqref{candidateoptimalcontrol} shows an important feature which is not present in the single regime case. Indeed, $\nu^*$ of \eqref{candidateoptimalcontrol} jumps at the moments of regime shifts from state $2$ to state $1$, thus implying a lump-sum extraction at those instants. This fact is not observed in \eqref{extraction1reg} where a jump can happen only at initial time. We also refer to the detailed discussion in \cite{Guoetal}.

It is also interesting to see how the presence of regime shifts is reflected into the optimal extraction boundaries. We study this in case (I) (i.e.\ for a strictly convex running cost function), and our findings are illustrated in Figure \ref{fig:2}.
\begin{figure}[!ht]
\centering
\includegraphics[scale=0.5]{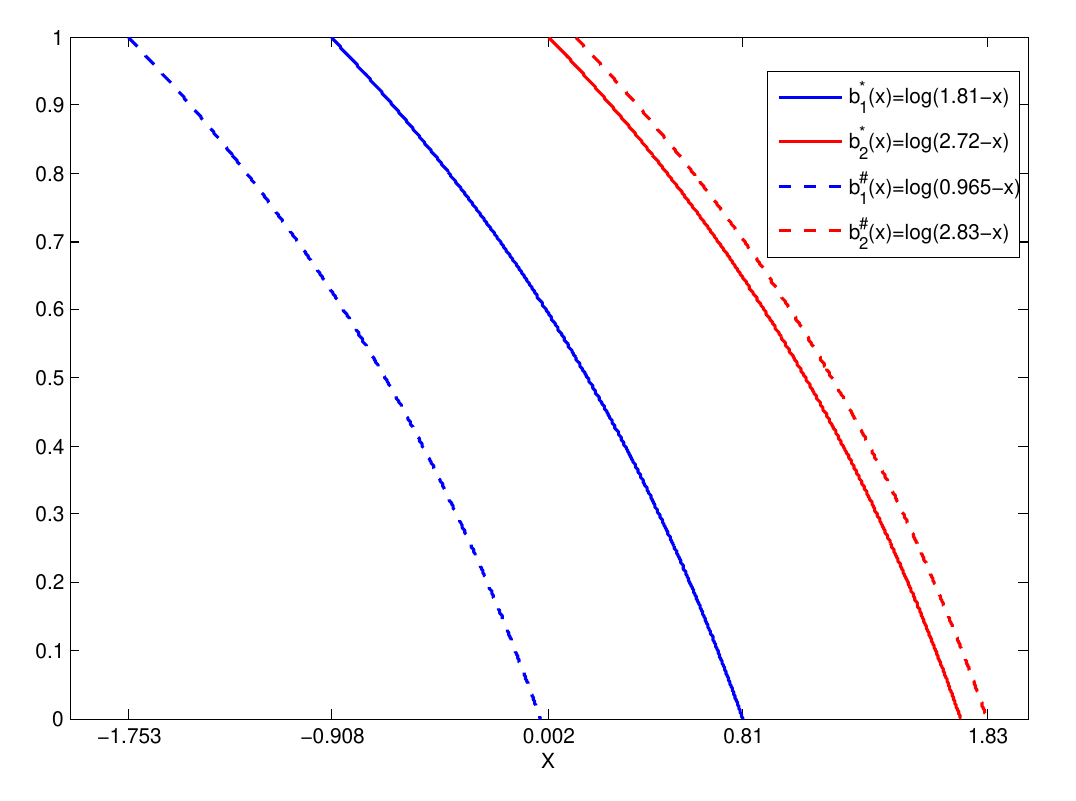}
\caption{\small The dashed curve $b^{\scriptstyle{\#}}_i(x)$, $i=1,2$, is the optimal extraction boundary \eqref{xcancello} of the single regime case when the volatility is $\sigma_i$. The solid curves are the optimal extraction boundaries $(b^*_1,b^*_2)$ when there is regime switching in the spot price process. To generate this plot with Matlab we have taken $f(y)=\frac{1}{3}(e^y-1)$ and with $\sigma_1=0.38$, $\sigma_2=1.9$, $\lambda_1=1.7$, $\lambda_2=0.44$, $\rho=1/3$ and $c=0.5$.}
\label{fig:2}
\end{figure}
There we take the strictly convex running cost $f(y)=\frac{1}{3}(e^y-1)$, and we plot the optimal boundaries in the case of regime switching, $b^*_i$, $i=1,2$ (solid curves), and in the case of a single regime, $b^{\scriptstyle{\#}}_i$ with volatility $\sigma_i$ (dashed curves), $i=1,2$. Taking $\sigma_1 < \sigma_2$ we observe, that under macroeconomic cycles, the value at which the reserve level should be kept is higher than the one at which it would be kept if the volatility were always $\sigma_1$. On the other hand, the value at which the reserve level should be maintained when business cycles are present, is lower than the one at which it would be kept if the volatility were always $\sigma_2$. To some extent, this fact can be thought of as an \emph{average effect} of the regime switching. For example, if the market volatility assumes at any time the highest value possible (i.e.\ it is always equal to $\sigma_2$), then the company would be more reluctant to extract and sell the commodity in the spot market relative to the case in which the volatility could jump to the lower value $\sigma_1$. A symmetric argument applies to explain $b^{\scriptstyle{\#}}_1 < b^{*}_i$, $i=1,2$.

\bigskip

\textbf{Acknowledgments.} We thank two anonymous Referees and an anonymous Associate Editor for their pertinent comments which helped a lot to improve previous versions of this paper. The first named author thanks Maria B.\ Chiarolla for having introduced him to the literature on optimal extraction problems under regime switching.


\appendix
\section{Some Proofs from Section \ref{sec:OS}}
\label{someproofs}
\renewcommand{\theequation}{A-\arabic{equation}}

\noindent\emph{Proof of Proposition \ref{preliminary:OS}}
\vspace{0.15cm}
The first claim immediately follows by taking the admissible $\tau=0$. 
As for the second property, let $\tau$ be an $\mathbb{F}$-stopping time and notice that by an integration by parts we can write 
\beq
\label{boundOS-1}
e^{-\rho \tau}(X_{\tau}-\theta(y)) = (x- \theta(y)) - \int_0^{\tau} \rho e^{-\rho s} \big(X_s-\theta(y)\big) ds + \int_0^{\tau} e^{-\rho s} \sigma_{\varepsilon_s}dW_s.
\eeq
Denoting $M_t:=\int_0^{t} e^{-\rho s} \sigma_{\varepsilon_s}dW_s$, $t\geq 0$, and recalling the boundedness of $\sigma_{\varepsilon_{\cdot}}$, $M$ is uniformly bounded in $L^2(\Omega,\mathbb{P}_{(x,i)})$, and therefore $\mathbb{P}_{(x,i)}$-uniformly integrable. Hence, taking expectations in \eqref{boundOS-1}, applying the optional stopping theorem (see Theorem 3.2 in \cite{RevuzYor}), and then taking absolute values we obtain
\begin{align}
\label{boundOS-2}
& \Big|\mathbb{E}_{(x,i)}\Big[e^{-\rho \tau}(X_{\tau}-\theta(y))\Big]\Big| \leq |x| + |\theta(y)| + \mathbb{E}_{(x,i)}\bigg[\int_0^{\infty} \rho e^{-\rho s} \big|X_s-\theta(y)\big| ds\bigg] \nonumber \\
& \leq 2(|x| + |\theta(y)|) + \int_0^{\infty} \rho e^{-\rho s}\mathbb{E}_{(x,i)}\bigg[\Big|\int_0^s \sigma_{\varepsilon_u}dW_u\Big|^2\bigg]^{\frac{1}{2}} ds \\
& \leq 2(|x| + |\theta(y)|) + (\sigma_1^2 \vee \sigma_2^2)^{\frac{1}{2}}\int_0^{\infty} \rho \sqrt{s} e^{-\rho s} ds \nonumber \leq K(y)(1 + |x|), 
\nonumber
\end{align}
for some $K(y)>0$. Equation \eqref{dyn:X}, Tonelli's Theorem and H\"older's inequality imply the second step above, whereas the third step is guaranteed by It\^o's isometry. The second claim of the proposition then easily follows from \eqref{boundOS-2}.
\ep

\vspace{0.25cm}

\noindent\emph{Proof of Theorem \ref{candidate-w}}
\vspace{0.15cm}

\emph{Step 1.}\, The fact that $w(\cdot, i;y) \in C^1(\mathbb{R})$ for $i=1,2$ follows by construction. It is also easy to verify from \eqref{eq:candidatevalue1} and \eqref{eq:candidatevalue2} that $w(\cdot, i;y)$, $i=1,2$, grows at most linearly and that $w_{xx}(\cdot,i;y)$ is bounded on any compact subset of $\mathbb{R}$.
\vspace{0.25cm}

We now show that $(w(x,1;y), w(x,2;y), x^*_1(y), x^*_2(y))$ solve free-boundary problem \eqref{FBP-1-a}-\eqref{FBP-1-c}.
Since $(w(x,1;y), w(x,2;y), x^*_1(y), x^*_2(y))$ satisfy \eqref{FBP-1-a} and \eqref{FBP-1-b} by construction, then it suffices to prove that also \eqref{FBP-1-c} is fulfilled. This part of the proof requires several estimates and it is organized in the next steps. In particular, \emph{Step 2}, \emph{Step 3} and \emph{Step 4} below are devoted to show that $w(x,i;y) \geq x -\theta(y)$ for $x \in \mathbb{R}$ and $i=1,2$. On the other hand, \emph{Step 5} shows that $\tfrac{1}{2}\sigma_i^2 w_{xx}(x,i;y)-\rho w(x,i;y) + \lambda_i(w(x,3-i;y)-w(x,i;y)) \leq 0$ for a.e.\ $x \in \mathbb{R}$ and for any $i=1,2$. 
\vspace{0.25cm}

\emph{Step 2.}\, Here we show that $w(x,1;y) \geq x - \theta(y)$ for any $x \in \mathbb{R}$. This is clearly true with equality by \eqref{eq:candidatevalue1} for any $x\geq x^*_1(y)$. To prove the claim when $x < x^*_1(y)$ we show that $w(\cdot,1;y)$ is convex therein. Indeed such property, together with the fact that $w_x(x^*_1(y),1;y) - 1 =0$, implies that $w_x(x,1;y) - 1 \leq 0$ for any $x < x^*_1(y)$. Hence, $w(x,1;y) \geq x - \theta(y)$ for $x < x^*_1(y)$ since also $w(x^*_1(y),1;y) - (x^*_1(y)-\theta(y)) =0$.

To complete, we thus need to show that $w(\cdot,1;y)$ is convex on $x < x^*_1(y)$. That is accomplished in the following. For any $x<x^*_1(y)$ we have from \eqref{eq:candidatevalue1}
\beq
\label{w1convex-1}
w_{xx}(x,1;y)(\alpha_4-\alpha_3)=\alpha_3^2(\alpha_4(x_1^*(y)-\theta(y))-1)e^{\alpha_3(x-x_1^*(y))}
+\alpha_4^2(1-\alpha_3(x_1^*(y)-\theta(y)))e^{\alpha_4(x-x_1^*(y))},
\eeq
and we want to prove that $w_{xx}(x,1;y)\geq 0$. To this end notice that some algebra gives
\begin{equation}
\label{w1convex-3}
\alpha_3^2(\alpha_4(x_1^*(y)-\theta(y))-1) +\alpha_4^2(1-\alpha_3(x_1^*(y)-\theta(y)))=(\alpha_4-\alpha_3)\Big[\alpha_4+\alpha_3-\alpha_3\alpha_4(x_1^*(y)-\theta(y))\Big],
\end{equation}
and also
\beq
\label{ex-ass-1}
\displaystyle  - \frac{1}{a_3}\Big(\frac{\rho}{\rho+\lambda_2}+ a_4\Big) - \frac{1}{\alpha_3} \leq \frac{1}{\alpha_4}. 
\eeq
Then recall that $x_1^*(y)-\theta(y) = z_1^*$, use the upper bound for $z^*_1$ given in \eqref{w1convex-2-a} and \eqref{ex-ass-1} into \eqref{w1convex-3}, to obtain $(\alpha_4-\alpha_3)[\alpha_4+\alpha_3-\alpha_3\alpha_4(x_1^*(y)-\theta(y))] \geq 0$.
By \eqref{w1convex-3} the latter implies that
$$\alpha_4^2(1-\alpha_3(x_1^*(y)-\theta(y))) \geq - \alpha_3^2(\alpha_4(x_1^*(y)-\theta(y))-1),$$
which substituted back into \eqref{w1convex-1} yields
\beq
\label{w1convex-4}
w_{xx}(x,1;y)(\alpha_4-\alpha_3) \geq \alpha_3^2(\alpha_4(x_1^*(y)-\theta(y))-1)\Big[e^{\alpha_3(x-x_1^*(y))}
-e^{\alpha_4(x-x_1^*(y))}\Big].
\eeq
But now the right hand-side of \eqref{w1convex-4} is nonnegative due to \eqref{w1convex-2-a}, \eqref{ex-ass-1}, and the fact that $\alpha_3 < \alpha_4$ but $x<x^*_1(y)$. Hence $w_{xx}(x,1;y)\geq 0$ for any $x < x^*_1(y)$, and therefore $w(\cdot,1;y)$ is convex on that region.
\vspace{0.25cm}

\emph{Step 3.}\, In this step we prove that $w(x^*_1(y),2;y) \geq x^*_1(y)-\theta(y)$ and $w_{x}(x^*_1(y),2;y) \leq 1$. These estimates will be needed in the next step to show that $w(x,2;y) \geq x-\theta(y)$ for any $x\in \mathbb{R}$. 

From \eqref{eq:candidatevalue2} and using that $B_3(y)=\frac{\Phi_1(\alpha_3)}{\lambda_1}A_3(y)$, $B_4(y)=\frac{\Phi_1(\alpha_4)}{\lambda_1}A_4(y)$, with $A_3(y)$ and $A_4(y)$ as in \eqref{A3-A4}, one easily finds
\beq
\label{w2x1}
w(x^*_1(y),2;y)=\frac{\Phi_1(\alpha_3)[\alpha_4(x_1^*(y)-\theta(y))-1]}{\lambda_1(\alpha_4-\alpha_3)}
+\frac{\Phi_1(\alpha_4)[1-\alpha_3(x_1^*(y)-\theta(y))]}{\lambda_1(\alpha_4-\alpha_3)}\nonumber
\eeq 
and
\beq
\label{w2xx1}
w_x(x^*_1(y),2;y)=\frac{\alpha_3\Phi_1(\alpha_3)[\alpha_4(x_1^*(y)-\theta(y))-1]}{\lambda_1(\alpha_4-\alpha_3)}
+\frac{\alpha_4\Phi_1(\alpha_4)[1-\alpha_3(x_1^*(y)-\theta(y))]}{\lambda_1(\alpha_4-\alpha_3)}.\nonumber
\eeq
Recalling that $\Phi_i(z)= -\frac{1}{2}\sigma_i^2 z^2 + \rho + \lambda_i$, $i=1,2$, a simple calculation yields
\beq
\label{w2x1-bis}
w(x^*_1(y),2;y)=\frac{-\frac{1}{2}\sigma_1^2(\alpha_3+\alpha_4)+(x_1^*(y)-\theta(y))
(\frac{1}{2}\sigma_1^2\alpha_3\alpha_4+\rho+\lambda_1)}{\lambda_1}
\eeq
and
\beq
\label{w2xx1-bis}
w_x(x^*_1(y),2;y)=\frac{\alpha_4\Phi(\alpha_4)-\alpha_3\Phi(\alpha_3)}{\lambda_1(\alpha_4-\alpha_3)}
+\frac{\alpha_3\alpha_4\sigma_1^2(x_1^*(y)-\theta(y))(\alpha_4+\alpha_3)}{2\lambda_1}.
\eeq
It is now matter of algebraic manipulations to show that 
\beq
\label{1}
\frac{\sigma_1^2(\alpha_3+\alpha_4)}{\sigma_1^2\alpha_3\alpha_4+2\rho}
= -\frac{a_2}{a_1+\frac{\rho}{\rho+\lambda_2}},
\eeq
and that
\beq
\label{2}
\displaystyle -\frac{\frac{\rho}{\rho+\lambda_2}+a_4}{a_3}
= \frac{2\lambda_1}{\alpha_3\alpha_4\sigma_1^2(\alpha_4+\alpha_3)}\Big[1+\frac{\alpha_3\Phi_1(\alpha_3)-\alpha_4\Phi_1(\alpha_4)}{\lambda_1(\alpha_4 -\alpha_3)}\Big].
\eeq
Then recalling that $x^*_1(y)-\theta(y)=z^*_1$, by \eqref{w1convex-2-a}, \eqref{1} and \eqref{2} we obtain
\begin{equation}
\label{ess-2}
\frac{\sigma_1^2(\alpha_3+\alpha_4)}{\sigma_1^2\alpha_3\alpha_4+2\rho}
\leq  x_1^*(y)-\theta(y)
\leq \frac{2\lambda_1}{\alpha_3\alpha_4\sigma_1^2(\alpha_4+\alpha_3)}\left[1+\frac{\alpha_3\Phi(\alpha_3)-\alpha_4\Phi(\alpha_4)}{\lambda_1(\alpha_4 -\alpha_3)}\right].
\end{equation}
By using the inequality on the left hand-side of \eqref{ess-2} in \eqref{w2x1-bis}, and the inequality on the right hand-side of \eqref{ess-2} in \eqref{w2xx1-bis}, we find $w(x^*_1(y),2;y) \geq x^*_1(y)-\theta(y)$ and $w_x(x^*_1(y),2;y) \leq 1$, respectively.
\vspace{0.25cm}


\emph{Step 4.}\, We now show that $w(x,2;y)\geq x-\theta(y)$ for $x < x^*_2(y)$ (and therefore for any $x\in \mathbb{R}$ due to the second of \eqref{FBP-1-b}). 

On $x \in (-\infty, x^*_1(y)) \cup (x^*_1(y), x^*_2(y))$ one has from \eqref{FBP-1-a}
\beq
\label{check-1}
\frac{1}{2}\sigma_2^2 w_{xx}(x,2;y) + \lambda_2(w(x,1;y) - w(x,2;y)) - \rho w(x,2;y) =0.
\eeq
Setting $\widehat{w}(x,i;y):=w(x,i;y)-(x-\theta(y))$, $i=1,2$, it follows that on $(-\infty, x^*_1(y)) \cup (x^*_1(y), x^*_2(y))$
\beq
\label{check-2}
\frac{1}{2}\sigma_2^2\widehat{w}_{xx}(x,2;y) + \lambda_2(\widehat{w}(x,1;y) - \widehat{w}(x,2;y)) - \rho \widehat{w}(x,2;y) = \rho(x-\theta(y)).
\eeq
We now show that $\widehat{w}(x,2;y) \geq 0$ separately in the two cases: \emph{(i)} $x \in (-\infty, x^*_1(y))$; and \emph{(ii)} $x \in (x^*_1(y), x^*_2(y))$. 
\vspace{0.12cm}

\emph{(i)}\ For $x \in (-\infty, x^*_1(y))$ we can differentiate \eqref{check-2} once more with respect to $x$ so to obtain
$$
\frac{1}{2}\sigma_2^2\widehat{w}_{xxx}(x,2;y) + \lambda_2(\widehat{w}_x(x,1;y) - \widehat{w}_x(x,2;y)) - \rho \widehat{w}_x(x,2;y) = \rho.
$$
Setting $\tau_1:=\inf\{t\geq 0: (X,\varepsilon) \notin \mathcal{D}_1\}$ $\mathbb{P}_{(x,i)}$-a.s., where $\mathcal{D}_1:=\{(x,i) \in \mathbb{R} \times \{1,2\}: x < x^*_1(y)\}$, an application of It\^o's formula (possibly with a standard localization argument) leads to
\begin{eqnarray}
\label{check-3}
\hspace{-0.5cm} \widehat{w}_{x}(x,2;y)& \hspace{-0.25cm} =  \hspace{-0.25cm} & \mathbb{E}_{(x,i)}\bigg[e^{-\rho \tau_1}\widehat{w}_x(X_{\tau_1}, \varepsilon_{\tau_1};y) - \int_0^{\tau_1} e^{-\rho s}\rho ds \bigg] \leq \mathbb{E}_{(x,i)}\Big[e^{-\rho \tau_1}\widehat{w}_x(X_{\tau_1}, \varepsilon_{\tau_1};y)\Big] \nonumber \\
& \hspace{-0.25cm} =  \hspace{-0.25cm} & \mathbb{E}_{(x,i)}\Big[e^{-\rho \tau_1}\widehat{w}_x(X_{\tau_1}, \varepsilon_{\tau_1};y)\mathds{1}_{\{\varepsilon_{\tau_1} = 1\}}\Big] + \mathbb{E}_{(x,i)}\Big[e^{-\rho \tau_1}\widehat{w}_x(X_{\tau_1}, \varepsilon_{\tau_1};y)\mathds{1}_{\{\varepsilon_{\tau_1} = 2\}}\Big],
\end{eqnarray}
for any $x < x^*_1(y)$. 

Recall now that $\widehat{w}_{x}(x^*_1(y),1;y) = {w}_{x}(x^*_1(y),1;y) - 1 =0$, and that by \emph{Step 3} $\widehat{w}_{x}(x^*_1(y),2;y) = {w}_{x}(x^*_1(y),2;y) - 1\leq 0$. Then the fact that $\tau_1 < +\infty$ $\mathbb{P}_{(x,i)}$-a.s.\ (by the recurrence property of $(X,\varepsilon)$; see (i) of Theorem 4.4 of \cite{ZhuYin} with $k>0$, $\alpha \in (0,1)$, $c_1=c_2$ therein) allows to conclude from \eqref{check-3} that $\widehat{w}_{x}(x,2;y) \leq 0$ for any $x < x^*_1(y)$. In turn this implies $w(x,2;y) \geq x-\theta(y)$ for any $x < x^*_1(y)$ since $w(x^*_1(y),2;y) \geq x^*_1(y)-\theta(y)$ again by the results of \emph{Step 3}.
\vspace{0.12cm}

\emph{(ii)}\ Take now $x \in (x^*_1(y), x^*_2(y))$ and define $\tau_{1,2}:=\inf\{t\geq 0: (X,\varepsilon) \notin \mathcal{D}_{1,2}\}$ $\mathbb{P}_{(x,i)}$-a.s., where $\mathcal{D}_{1,2}:=\{(x,i) \in \mathbb{R} \times \{1,2\}: x^*_1(y) < x < x^*_2(y)\}$.  By arguments similar to those employed in \emph{(i)}, but now using that $\widehat{w}_{x}(x^*_2(y),2;y) = 0$ and $\widehat{w}_{x}(x^*_1(y),2;y) \leq 0$ (cf.\ \emph{Step 3}), and that $\widehat{w}_x(x^*_2(y),1;y) = 0 = \widehat{w}_x(x^*_1(y),1;y)$ by construction, we obtain $\widehat{w}_x(x,2;y) \leq 0$ for any $x \in (x^*_1(y), x^*_2(y))$.
Hence $\widehat{w}(x,2;y) \geq 0$ for any $x \in (x^*_1(y), x^*_2(y))$ since $\widehat{w}(x^*_2(y),2;y) = 0$. 
\vspace{0.12cm}

By combining \emph{(i)} and \emph{(ii)} we have thus proved that ${w}(x,2;y) \geq x - \theta(y)$ for any $x \in (-\infty, x^*_1(y)) \cup (x^*_1(y), x^*_2(y))$. However, we already know by \emph{Step 3} that ${w}(x^*_1(y),2;y) \geq x^*_1(y) - \theta(y)$, and therefore we can conclude that $w(x,2;y)\geq x-\theta(y)$ for any $x < x^*_2(y)$.
\vspace{0.25cm}

Steps \emph{2}, \emph{3} and \emph{4} above show that $w(x,i;y) \geq x -\theta(y)$ for $x \in \mathbb{R}$ and $i=1,2$. We now turn to prove that one also has $\tfrac{1}{2}\sigma_i^2 w_{xx}(x,i;y)-\rho w(x,i;y) + \lambda_i(w(x,3-i;y)-w(x,i;y)) \leq 0$ for a.e.\ $x \in \mathbb{R}$ and $i=1,2$.
\vspace{0.25cm}

\emph{Step 5.}\,\emph{(i)}\ We start by showing that 
\beq
\label{ineq-1}
\tfrac{1}{2}\sigma_2^2 w_{xx}(x,2;y)-\rho w(x,2;y) + \lambda_2(w(x,1;y)-w(x,2;y)) \leq 0
\eeq
for a.e.\ $x \in \mathbb{R}$. This is true with equality for any $x < x^*_2(y)$ by construction. For $x > x^*_2(y)$ we have $w(x,1;y)=x-\theta(y)=w(x,2;y)$, so that \eqref{ineq-1} reads $- \rho (x-\theta(y)) \leq 0$. But now the latter inequality holds since $\rho > 0$ and $x^*_2(y) > \theta(y)$ by Corollary \ref{cor:boundary-system}.
\vspace{0.12cm}

\emph{(ii)}\ We now check that one also has
\beq
\label{ineq-2}
\tfrac{1}{2}\sigma_1^2 w_{xx}(x,1;y)-\rho w(x,1;y) + \lambda_1(w(x,2;y)-w(x,1;y)) \leq 0
\eeq
for a.e.\ $x \in \mathbb{R}$. Again, it suffices to show that the previous is true for $x > x^*_1(y)$, as it is verified with equality by construction on $(-\infty,x^*_1(y))$. 

If $x > x^*_2(y)$ then $w(x,2;y)= x - \theta(y)=w(x,1;y)$ and \eqref{ineq-2} holds since $\rho > 0$ and $x^*_2(y) > \theta(y)$ by Corollary \ref{cor:boundary-system}. 

To complete the proof we consider the case $x\in (x_1^*(y), x^*_2(y))$. On such an interval we have again $w(x,1;y)=x-\theta(y)$, and therefore \eqref{ineq-2} is verified on $(x_1^*(y), x^*_2(y))$ if 
\beq
\label{ineq-3}
w(x,2;y) \leq \frac{\rho + \lambda_1}{\lambda_1} w(x,1;y).
\eeq
In \emph{Step 4} we have shown that $w_x(x,2;y) - 1 \leq 0$ for any $x\in (x_1^*(y), x^*_2(y))$, from which one has 
\begin{eqnarray*}
& w(x,2;y) - w(x,1;y) = w(x,2;y) - (x - \theta(y)) \leq w(x_1^*(y),2;y) - (x_1^*(y) - \theta(y)) \nonumber \\
& = w(x_1^*(y),2;y) - w(x_1^*(y),1;y),
\end{eqnarray*}
where the fact that $w(x,1;y) = x - \theta(y)$ for any $x \geq x_1^*(y)$ has been used. Therefore on $(x_1^*(y), x^*_2(y))$
\beq
\label{ineq-4}
w(x,2;y)\leq w(x_1^*(y),2;y)-w(x_1^*(y),1;y)+ w(x,1;y),
\eeq

However, by convexity of $w(\cdot,1;y)$ proved in \emph{Step 2} one has 
$$-\rho w(x,1;y) + \lambda_1(w(x,2;y)-w(x,1;y)) \leq \tfrac{1}{2}\sigma_1^2 w_{xx}(x,1;y)-\rho w(x,1;y) + \lambda_1(w(x,2;y)-w(x,1;y)) = 0$$ 
for any $x < x^*_1(y)$, and this yields 
\beq
\label{eq:ineq-5}
w(x,2;y) \leq \frac{\rho+\lambda_1}{\lambda_1} w(x,1;y),\quad x < x^*_1(y).
\eeq
Then, taking limits as $x \uparrow x^*_1(y)$ we get from \eqref{eq:ineq-5} and continuity of $w(\cdot,i;y)$
\beq
\label{eq:ineq-5-bis}
w(x_1^*(y),2;y) \leq \frac{\rho+\lambda_1}{\lambda_1} w(x_1^*(y),1;y),
\eeq
and we conclude from \eqref{ineq-4} and \eqref{eq:ineq-5-bis} that for any $x \in (x_1^*(y), x^*_2(y))$
$$
w(x,2;y)\leq \frac{\rho+\lambda_1}{\lambda_1} w(x^*_1(y),1;y)-w(x^*_1(y),1;y) + w(x,1;y)\leq \frac{\rho+\lambda_1}{\lambda_1}w(x,1;y),
$$
where the fact that $w(x^*_1(y),1;y) = x^*_1(y)-\theta(y) \leq (x-\theta(y)) =w(x,1;y)$ for any $x >  x^*_1(y)$ implies the last step. Hence \eqref{ineq-3} holds on $(x_1^*(y), x^*_2(y))$, and therefore also \eqref{ineq-2} is satisfied on that interval. This completes the proof.
\ep

\vspace{0.25cm}

\noindent\emph{Proof of Theorem \ref{thm:verifying}}
\vspace{0.15cm}

\emph{Step 1.} Fix $(x,i) \in \mathbb{R} \times \{1,2\}$, let $\tau$ be an arbitrary $\mathbb{P}_{(x,i)}$-a.s.\ finite stopping time, and set $\tau_R:=\inf\{t\geq 0: X_t \notin (-R,R)\}$ $\mathbb{P}_{(x,i)}$-a.s.\ for $R>0$. Then, let $0 \leq \eta_1 < \eta_2 < ... < \eta_{N)} \leq \tau \wedge \tau_R$ be the random times of jumps of $\varepsilon$ in the interval $[0,\tau \wedge \tau_R)$ (clearly, the number $N$ of these jumps is random as well)  and, given the regularity of $w(\cdot,i;y)$ for any $i=1,2$ (cf.\ Theorem \ref{candidate-w}), apply It\^o-Tanaka's formula (see, e.g., \cite{RevuzYor}, Chapter VI, Proposition 1.5, Corollary 1.6 and following Remarks) between consecutive jumps of $\varepsilon$ from time $0$ up to time $\tau \wedge \tau_R$. Piecing together all the terms as in the proof of Lemma 3 at p.\ 104 of \cite{Sk} (see also Lemma 2.4 and its idea of proof in \cite{YinXi}) we find
\begin{eqnarray}
\label{verifico-1}
w(x, i;y) & \hspace{-0.25cm} = \hspace{-0.25cm} & \mathbb{E}_{(x,i)}\bigg[e^{-\rho (\tau\wedge \tau_R)} w(X_{\tau\wedge \tau_R}, \varepsilon_{\tau\wedge \tau_R};y) - \int_0^{\tau\wedge \tau_R}e^{-\rho s}(\mathcal{G}-\rho)w(X_s, \varepsilon_s; y)ds\bigg] \nonumber \\
& \hspace{-0.25cm} \geq \hspace{-0.25cm} & \mathbb{E}_{(x,i)}\Big[e^{-\rho(\tau\wedge \tau_R)} w(X_{\tau\wedge \tau_R},\varepsilon_{\tau\wedge \tau_R}; y)\Big] \geq \mathbb{E}_{(x,i)}\Big[e^{-\rho (\tau\wedge \tau_R)}(X_{\tau \wedge \tau_R}-\theta(y))\Big].
\end{eqnarray}
In \eqref{verifico-1} we have used that $w$ solves free-boundary problem \eqref{FBP-1-a}-\eqref{FBP-1-c} (cf.\ Theorem \ref{candidate-w}), and the fact that the stochastic integral over the interval $[0,\tau\wedge \tau_R)$ vanishes under expectation since $w_x$ is bounded for $(x,i,y) \in [-R,R] \times \{1,2\} \times [0,1]$. 

But now $\{e^{-\rho (\tau\wedge \tau_R)} X_{\tau \wedge \tau_R},\, R>0\}$ is a $\mathbb{P}_{(x,i)}$-uniformly integrable family by Lemma \ref{lemma:new} in Appendix \ref{app}, hence observing that if $R\uparrow \infty$ one has $\tau \wedge \tau_R \uparrow \tau$ a.s.\ by regularity of $(X,\varepsilon)$ (cf.\ \cite{ZhuYin}, Section 3.1), we can take limits as $R\uparrow \infty$ in \eqref{verifico-1}, invoke Vitali's convergence theorem and obtain 
$$w(x, i;y) \geq \mathbb{E}_{(x,i)}\Big[e^{-\rho \tau}(X_{\tau}-\theta(y))\Big].$$
Since $\tau$ was arbitrary, $ w(x, i ;y)\geq \sup_{\tau \geq 0}\mathbb{E}_{(x,i)}[e^{-\rho \tau}(X_{\tau}-\theta(y))\Big] = u(x,i;y)$.
\vspace{+8pt}

\emph{Step 2.} To prove the reverse inequality, i.e.\ $w(x,i;y) \leq u(x,i;y)$, take $\tau=\tau^*$, in the previous arguments and notice that one has $(\mathcal{G}-\rho)w(x, i; y) =0$ on $\mathcal{C}$. Then taking limits as $R\uparrow \infty$ one finds
\begin{eqnarray}
\label{taustar0}
w(x,i;y) &\hspace{-0.25cm} = \hspace{-0.25cm}& \mathbb{E}_{(x,i)}\Big[e^{-\rho \tau^*}w(X_{\tau^*},\varepsilon_{\tau^*}; y)\Big] = \mathbb{E}_{(x,i)}\Big[e^{-\rho \tau^*}(X_{\tau^*}-\theta(y))\Big],
\end{eqnarray}
where the last equality follows from the fact that $\tau^* < +\infty$ $\mathbb{P}_{(x,i)}$-a.s.\ by recurrence of $(X, \varepsilon)$ (cf.\ Theorem 4.4 in \cite{ZhuYin}). Therefore $w(x,i;y) \leq u(x,i;y)$, whence $w(x,i;y) = u(x,i;y)$ and optimality of $\tau^*$.
\ep

\section{Some Auxiliary Results}
\label{app}
\renewcommand{\theequation}{B-\arabic{equation}}

\begin{lemma}
\label{4thordereq}
For $i=1,2$ and $\alpha \in \mathbb{R}$, let $\Phi_i(\alpha):=-\frac{1}{2}\sigma_i^2\alpha^2+\rho+\lambda_i$. Then there exist unique $\alpha_1 < \alpha_2 <0 <\alpha_3 <\alpha_4$ satisfying the fourth-order equation 
\beq
\label{4th}
\Phi_1(\alpha)\Phi_2(\alpha) - \lambda_1\lambda_2=0.
\eeq
\end{lemma}
\begin{proof}
We provide here a proof of this claim in our setting for the sake of completeness (see also \cite{Guo01}, Remark 2.1, and \cite{CadSot}, Lemma 3.1 for related results). 
Using the definition of $\Phi_i$, $i=1,2$, equation \eqref{4th} reads
$$\frac{1}{4}\sigma_1^2\sigma_2^2\alpha^4 - \Big[\frac{1}{2}\sigma_1^2(\rho + \lambda_2) + \frac{1}{2}\sigma_2^2(\rho + \lambda_1)\Big]\alpha^2 + (\rho + \lambda_1)(\rho + \lambda_2) - \lambda_1\lambda_2=0,$$
and letting
$$a_o:=\frac{1}{4}\sigma_1^2\sigma_2^2, \quad b_o:=\frac{1}{2}\sigma_1^2(\rho + \lambda_2) + \frac{1}{2}\sigma_2^2(\rho + \lambda_1), \quad c_o:=(\rho + \lambda_1)(\rho + \lambda_2) - \lambda_1\lambda_2,$$
one can check that
$$b_o^2 - 4 a_o c_o = \Big[\frac{\sigma_1^2(\rho + \lambda_2) - \sigma_2^2(\rho + \lambda_1)}{2}\Big]^2 + \lambda_1\lambda_2\sigma_1^2\sigma_2^2 > 0.$$
Hence there exists two solutions $\beta_1$ and $\beta_2$ to the second-order equation $a_o \beta^2 - b_o\beta + c_o =0$, and they are such that $0 < \beta_2 < \beta_1$ since $a_o c_o > 0$. It thus follows that
$$-\alpha_1: = \sqrt{\beta_1} =:\alpha_4, \qquad -\alpha_2: = \sqrt{\beta_2} =:\alpha_3$$
solve \eqref{4th} and satisfy $\alpha_1 < \alpha_2 <0 <\alpha_3 <\alpha_4$. 

\end{proof}

\begin{lemma}
\label{lemm:valuesai}
Let $a_i$, $i=1,2,3,4$, be defined as in \eqref{a1-a4}. Then one has $a_1 < 0$, $a_2>0$, $a_3 < 0$ and $a_4>0$.
\end{lemma}  
\begin{proof}
Noticing that $\Phi_i(\alpha)=-\frac{1}{2}\sigma_i^2 \alpha^2+\rho+\lambda_i$, $i=1,2$, is a strictly decreasing function of $\alpha$,  the fact that $\alpha_3 < \alpha_4$ imply $a_2>0$ and $a_3<0$.

As for $a_1$, recall that from \eqref{a1-a4} one has
\beq
\label{a1}
a_1=-\frac{\alpha_4\Phi_1(\alpha_3)-\alpha_3\Phi_1(\alpha_4)}{\lambda_1(\alpha_4-\alpha_3)}
+\frac{\lambda_2}{\rho+\lambda_2}.
\eeq
By using the explicit expression of $\Phi_i(\alpha)$, $i=1,2$, direct calculations lead to
\begin{equation}
\label{Phis-1}
\alpha_4\Phi_1(\alpha_3)-\alpha_3\Phi_1(\alpha_4)=
\Big(\frac{1}{2}\sigma_1^2\alpha_3\alpha_4+\rho+\lambda_1\Big)(\alpha_4-\alpha_3),
\end{equation}
which substituted into \eqref{a1} yields
\beq
\label{a1final}
a_1=-\frac{\frac{1}{2}\sigma_1^2\alpha_3\alpha_4+\rho+\lambda_1}{\lambda_1} + \frac{\lambda_2}{\rho+\lambda_2} < -\frac{\frac{1}{2}\sigma_1^2\alpha_3\alpha_4+\rho}{\lambda_1}<0.
\eeq

We conclude showing that $a_4>0$. It is matter of simple algebra to show that
\beq
\label{Phis-2}
\alpha_3\Phi_1(\alpha_3)-\alpha_4\Phi_1(\alpha_4)=
(\alpha_4-\alpha_3)\Big[\frac{1}{2}\sigma_1^2(\alpha_3\alpha_4+\alpha_3^2+\alpha_4^2)-(\rho + \lambda_1)\Big],
\eeq
which used in the expression for $a_4$ of \eqref{a1-a4} allows to write
\beq
\label{a4}
a_4=\frac{\frac{1}{2}\sigma_1^2(\alpha_3\alpha_4+\alpha_3^2+\alpha_4^2)-(\rho +\lambda_1)}{\lambda_1}
+\frac{\lambda_2}{\rho+\lambda_2}.
\eeq
Since $\alpha_3$ and $\alpha_4$ solve $\Phi_1(\alpha)\Phi_2(\alpha)=\lambda_1\lambda_2$, by Vieta's formulas we deduce that
\beq
\label{Vietaalpha34}
\alpha_3^2+\alpha_4^2=\frac{2\sigma_1^2(\rho+\lambda_2)+2\sigma_2^2
(\rho+\lambda_1)}{\sigma_1^2\sigma_2^2}.
\eeq
Noticing that $\alpha_3\alpha_4>0$, and using \eqref{Vietaalpha34} in \eqref{a4} we obtain
$$a_4>\frac{\frac{1}{2}\sigma_1^2(\alpha_3^2+\alpha_4^2)- (\rho +\lambda_1)}{\lambda_1}
>\frac{1}{\lambda_1}\Big[\frac{\sigma_1^2\sigma_2^2
(\rho+\lambda_1)}{\sigma_1^2\sigma_2^2}- (\rho + \lambda_1)\Big]=0,$$
thus completing the proof.
\end{proof}


\begin{lemma}
\label{lemma:new}
Fix $(x,i) \in \mathbb{R} \times \{1,2\}$, let $\tau$ be an arbitrary $\mathbb{P}_{(x,i)}$-a.s.\ finite stopping time, and for $R>0$ set $\tau_R:=\inf\{t\geq 0: X_t \notin (-R,R)\}$ $\mathbb{P}_{(x,i)}$-a.s. Then the family of random variables $\{e^{-\rho(\tau\wedge \tau_R)} X_{\tau\wedge \tau_R},\, R>0\}$ is $\mathbb{P}_{(x,i)}$-uniformly integrable. 
\end{lemma}
\begin{proof}
By an integration by parts we have due to \eqref{dyn:X}
$$e^{-\rho(\tau\wedge \tau_R)} X_{\tau\wedge \tau_R} = x - \int_0^{\tau\wedge \tau_R} \rho e^{-\rho s} X_s ds + \int_0^{\tau\wedge \tau_R} e^{-\rho s} \sigma_{\varepsilon_s} dW_s.$$
On the one hand, by H\"older's inequality and It\^o's isometry one has
\begin{eqnarray}
\label{boundUI-a}
\mathbb{E}_{(x,i)}\bigg[\int_0^{\infty} \rho e^{-\rho s} |X_s| ds\bigg] 
& \hspace{-0.25cm} \leq \hspace{-0.25cm} & |x| + \int_0^{\infty} \rho e^{-\rho s}\mathbb{E}_{(x,i)}\bigg[\Big|\int_0^s \sigma_{\varepsilon_u}dW_u\Big|^2\bigg]^{\frac{1}{2}} ds\\
& \hspace{-0.25cm} \leq \hspace{-0.25cm} & |x| + (\sigma_1^2 \vee \sigma_2^2)^{\frac{1}{2}}\int_0^{\infty} \rho \sqrt{s} e^{-\rho s} ds < \infty, \nonumber
\nonumber
\end{eqnarray}
for some $K>0$. Hence $\int_0^{\infty} \rho e^{-\rho s} |X_s| ds \in L^1(\Omega, \mathbb{P}_{(x,i)})$.
On the other hand, the continuous martingale $\{\int_0^{t} e^{-\rho s} \sigma_{\varepsilon_s} dW_s,\, t\geq 0\}$ is bounded in $L^2(\Omega, \mathbb{P}_{(x,i)})$ since $\mathbb{E}_{(x,i)}[|\int_0^{t} e^{-\rho s} \sigma_{\varepsilon_s} dW_s|^2] \leq (\sigma_1^2 \vee \sigma_2^2) \int_0^{\infty} e^{-2\rho s} ds$, and therefore (cf.\ \cite{RevuzYor}, Chapter IV, Proposition 1.23)
$$\mathbb{E}_{(x,i)}\bigg[\Big|\int_0^{\tau\wedge \tau_R} e^{-\rho s} \sigma_{\varepsilon_s} dW_s\Big|^2\bigg] = \mathbb{E}_{(x,i)}\bigg[\int_0^{\tau\wedge \tau_R} e^{-2\rho s} \sigma^2_{\varepsilon_s} ds \bigg] \leq (\sigma_1^2 \vee \sigma_2^2) \int_0^{\infty} e^{-2\rho s} ds, \quad R > 0.$$
Hence, the family $\{\big|\int_0^{\tau\wedge \tau_R} e^{-\rho s} \sigma_{\varepsilon_s} dW_s\big|,\, R>0\}$ is bounded in $L^2(\Omega, \mathbb{P}_{(x,i)})$ as well, thus uniformly integrable. This fact, together with \eqref{boundUI-a}, in turn imply uniform integrability of the family $\{e^{-\rho(\tau\wedge \tau_R)} X_{\tau\wedge \tau_R},\, R>0\}$.
\end{proof}


\begin{lemma}
\label{lemma:UI}
Let $(x,y,i)\in \mathcal{O}$ and denote by $\mathcal{T}$ the set of $\mathbb{F}$-stopping times. Then for any $\nu \in \mathcal{A}_y$, the families of random variables
$$\bigg\{\int_0^{\tau} e^{-\rho u}(X_u-c)d\nu_u,\,\,\,\tau\in \mathcal{T}\bigg\}\qquad \text{and} \qquad \bigg\{\int_0^{\tau} e^{-\rho u} f(Y^{\nu}_u) du,\,\,\,\tau\in \mathcal{T}\bigg\}$$
are $\mathbb{P}_{(x,y,i)}$-uniformly integrable.
\end{lemma}
\begin{proof}
We prove the uniform integrability of the first family of random variables by showing that it is uniformly bounded in $L^2(\Omega,\mathbb{P}_{(x,y,i)})$. Let $\tau$ be any given and fixed stopping time of $\mathbb{F}$, take any $\nu\in \mathcal{A}_y$, and notice that an integration by parts leads to
\beq
\label{UI-1}
\int_0^{\tau} e^{-\rho u}(X_u -c)d\nu_u = e^{-\rho \tau}(X_{\tau} -c)\nu_{\tau} +  \int_0^{\tau} \rho e^{-\rho u}(X_u -c)\nu_u du - \int_0^{\tau} e^{-\rho u} \nu_u\sigma_{\varepsilon_u}dW_u,
\eeq
where \eqref{dyn:X} has been employed. However we also have
\beq
\label{UI-1-BIS}
e^{-\rho \tau}(X_{\tau} -c)\nu_{\tau} = \nu_{\tau}\bigg[x - c e^{-\rho \tau} - \int_0^{\tau} \rho e^{-\rho u}X_u du + \int_0^{\tau} e^{-\rho u} \sigma_{\varepsilon_u}dW_u\bigg].
\eeq

Denoting by $K>0$ a suitable constant possibly depending on $x$ and $y$, but not on $\tau$, and that may change from line to line, we obtain from \eqref{UI-1} and \eqref{UI-1-BIS}
\begin{eqnarray}
\label{UI-2}
&& \Big|\int_0^{\tau} e^{-\rho u}(X_u -c)d\nu_u\Big|^2 \leq K\Big[ 1 + \int_0^{\infty} \rho e^{-\rho u}|X_u|^2 du + \Big|\int_0^{\tau} e^{-\rho u} \sigma_{\varepsilon_u}dW_u\Big|^2 \nonumber \\
&& \hspace{0.5cm} + \Big|\int_0^{\tau} e^{-\rho u} \nu_u\sigma_{\varepsilon_u}dW_u\Big|^2\Big]  \leq K \Big[ 1 + \int_0^{\infty} \rho e^{-\rho u}\Big|\int_0^u e^{-\rho s} \sigma_{\varepsilon_s}dW_s\Big|^2 du   \\
&& \hspace{0.5cm} + \Big|\int_0^{\tau} e^{-\rho u}\sigma_{\varepsilon_u}dW_u\Big|^2 + \Big|\int_0^{\tau} e^{-\rho u} \nu_u\sigma_{\varepsilon_u}dW_u\Big|^2\Big],\nonumber 
\end{eqnarray}
where the boundedness of $\nu \in \mathcal{A}_y$ has been exploited. In \eqref{UI-2} Jensen's inequality has been used in the first step for the integrals with respect to $\rho e^{-\rho u}du$, whereas the last step employs \eqref{dyn:X}. Taking expectations in \eqref{UI-2}, using It\^o's isometry, and noticing that $\sigma_{\varepsilon_t}^2 \leq \sigma_1^2 \vee \sigma_2^2$ a.s.\ and that any admissible control is bounded by one, we obtain 
\beq
\label{UI-3}
\mathbb{E}_{(x,y,i)}\bigg[\Big|\int_0^{\tau} e^{-\rho u}(X_u -c)d\nu_u\Big|^2\bigg] \leq K\Big[ 1 + (\sigma_1^2 \vee \sigma_2^2) \int_0^{\infty} \rho e^{-\rho u} (1+u)\,du\Big],
\eeq
which in turn proves the first claim.

Uniform integrability of the second family follows by noticing that for any $\mathbb{F}$-stopping time $\tau$ and any $\nu \in \mathcal{A}_y$ we have
$$0 \leq \int_0^{\tau} e^{-\rho u}f(Y^{\nu}_u)du  \leq \int_0^{\infty} e^{-\rho u}f(1)du \leq \frac{f(1)}{\rho},$$
where we have used the fact that $f(\cdot)$ is nonnegative and increasing, and that $Y^{\nu}_t \leq 1$ a.s.
\end{proof}


\end{document}